\documentclass[pdflatex,sn-mathphys-num]{sn-jnl}

\usepackage{graphicx}%
\usepackage{multirow}%
\usepackage{amsmath,amssymb,amsfonts}%
\usepackage{amsthm}%
\usepackage{mathrsfs}%
\usepackage[title]{appendix}%
\usepackage{xcolor}%
\usepackage{textcomp}%
\usepackage{manyfoot}%
\usepackage{booktabs}%
\usepackage{algorithm}%
\usepackage{algorithmicx}%
\usepackage{algpseudocode}%
\usepackage{listings}%

\theoremstyle{plain}
\newtheorem{theorem}{Theorem}[section]
\newtheorem{lemma}[theorem]{Lemma}
\newtheorem{proposition}[theorem]{Proposition}
\newtheorem{corollary}[theorem]{Corollary}

\theoremstyle{definition}
\newtheorem{assumption}[theorem]{Assumption}
\newtheorem{hypothesis}[theorem]{Hypothesis}

\theoremstyle{remark}
\newtheorem{remark}[theorem]{Remark}

\newcommand{\R}{\mathbb R}
\newcommand{\N}{\mathbb N}
\DeclareMathOperator{\sgn}{sgn}

\begin{document}

\title{Bilinear control of age--space structured populations}

\author[1]{\fnm{Jiguang} \sur{Yu}}\email{jyu678@bu.edu}
\equalcont{These authors contributed equally to this work as co-first authors.}

\author*[2]{\fnm{Louis Shuo} \sur{Wang}}\email{wang.s41@northeastern.edu}
\equalcont{These authors contributed equally to this work as co-first authors.}

\author*[3]{\fnm{Ye} \sur{Liang}}\email{ye-liang@uiowa.edu}

\affil[1]{\orgdiv{College of Engineering},
  \orgname{Boston University},
  \orgaddress{\city{Boston}, \postcode{02215}, \state{MA}, \country{USA}}}

\affil[2]{\orgdiv{Department of Mathematics},
  \orgname{Northeastern University},
  \orgaddress{\city{Boston}, \postcode{02115}, \state{MA}, \country{USA}}}

\affil[3]{\orgdiv{College of Engineering},
  \orgname{The University of Iowa},
  \orgaddress{\city{Iowa City}, \postcode{52242}, \state{IA},\country{USA}}}

\abstract{
We study constrained bilinear optimal control for nonlocal age--space
structured population equations with renewal boundary conditions and
endogenous surveillance feedback. The control acts as a coefficient in a mixed
transport--diffusion equation, while a scalar observable generated by the state
enters both the interior dynamics and the renewal law. This produces a
nonlinear closed-loop control-to-state map and a feedback-dependent adjoint
system. Using a characteristic mild formulation rather than a standard
Lions--Magenes argument, we establish closed-loop well-posedness and
Fr\'echet differentiability. We then derive the reduced and
feedback-corrected adjoint equations. The feedback derivative is identified as
a low-rank perturbation \(\ell_{\bar y,\bar u}(p)(t)\chi(a,x)\); in the
Volterra-kernel regime, the associated transfer operator is quasinilpotent,
yielding an explicit resolvent representation of the adjoint.
Finally, we prove first-order optimality conditions and decompose the
switching function into reduced and feedback-induced components.
}



\maketitle

\section{Introduction}
\label{sec:intro}
This paper studies constrained bilinear optimal control problems for nonlocal
age--space structured population equations with surveillance observations. The
control acts as a coefficient in the state equation, the state satisfies a
renewal boundary condition in the age variable, and the coefficients depend on
an endogenous feedback observable. This combination produces a nonlinear
control-to-state map and a singular adjoint equation in which the feedback
closure induces a low-rank perturbation. The mathematical theory of
age-structured population dynamics, with and without spatial diffusion, is
classical \cite{webb1985theory,iannelli1995mathematical,langlais1985nonlinear,webb2008population,langlais1988large,nag2025analysis,wang2025analysis1,kang2026principal,kang2025global,wang2025analysis}, as is its optimal control
\cite{yu2026pattern,anita2013analysis,barbu2012mathematical,lions1971optimal}. The class of systems considered here combines
four analytical mechanisms:
\[
\text{transport in age}
\quad+\quad
\text{spatial diffusion}
\quad+\quad
\text{nonlocal interaction}
\quad+\quad
\text{bilinear coefficient control}.
\]
Each is standard in isolation, but their simultaneous presence creates a
substantially different optimal-control problem. Age transport and renewal
boundary conditions destroy the purely parabolic structure of the state
equation; spatial diffusion provides only partial regularization; nonlocal
interaction terms produce nonlinear population coupling; and bilinear
coefficient control makes the control-to-state map nonlinear even when the
remaining dynamics are linear. The endogenous feedback closure introduces an
additional nonlocal term in the adjoint equation.

A guiding methodological point is that the age direction is hyperbolic: there
is no smoothing along characteristics $t-a=\mathrm{const}$, and spatial
diffusion regularizes only in $x$. Consequently the natural energy space does
not embed into continuous functions on $\overline Q$, and time continuity into
$X$ cannot be obtained from a Lions--Magenes argument
\cite{lions2012non}, because the energy estimate controls only the transport
derivative $D_{t,a}y=\partial_t y+\partial_a y$, not $\partial_t y$ separately.
We therefore (i) obtain time continuity into $X$ from the characteristic mild
representation and the regularizing spatial evolution family
\cite{pazy2012semigroups,acquistapace1987unified,wang2026algebraic}, and (ii) take the \emph{default} observation
model to be \emph{time-averaged} localized surveillance, which is continuous
with respect to strong $L^2(Q)$ convergence and yields a non-singular adjoint
source. Instantaneous fixed-time localized observations and pointwise
observations are treated as conditional extensions under additional regularity.

Let
\[
\mathcal I_a = (0,a_{\max}),\quad Q:=(0,T)\times\mathcal I_a\times\Omega ,
\quad
\Omega\subset\R^d,\quad d\in\{1,2,3\},
\]
and let
\[
X:=L^2(\mathcal I_a\times\Omega;\R^n),
\qquad
V:=L^2(\mathcal I_a;H^1(\Omega;\R^n)).
\]
The state is an $n$-component structured population density
$y=y(t,a,x)\in\R^n$, depending on time $t$, age or size $a$, and position $x$.
The prototype equation has the form
\begin{equation}
\label{eq:intro-state}
\partial_t y+\partial_a y
-\nabla_x\cdot(\Sigma(a,x)\nabla_x y)
+
\mathcal A(E(t),a,x)y
+
\mathcal N[y](t,a,x)y
+
\mathcal K(u)y
=
F
\quad\text{in }Q,
\end{equation}
with $\Sigma$ a uniformly elliptic diffusion tensor, $\mathcal A$ a feedback-dependent
coefficient matrix, $\mathcal N[y]y$ a nonlocal population interaction, and
$\mathcal K(u)y$ the bilinear control term. A typical nonlocal interaction is
\[
\mathcal N[y](t,a,x)
=
\int_{\mathcal I_a}\int_\Omega
\mathscr K(a,\alpha,x,\xi)y(t,\alpha,\xi)\,d\xi\,d\alpha .
\]
The renewal boundary condition is
\[
y(t,0,x)
=
\int_{\mathcal I_a}
\mathcal B(E(t),\alpha,x)y(t,\alpha,x)\,d\alpha ,
\]
and the feedback variable is the endogenous observable
\[
E(t)
=
\int_{\mathcal I_a}\int_\Omega
\chi(a,x)\cdot y(t,a,x)\,dx\,da .
\]
The control is constrained by box bounds:
\begin{equation}
\label{eq:intro-control-set}
U_{\rm ad}
:=
\left\{
u\in L^\infty(Q;\R^m):
u_{\min,\ell}\le u_\ell(t,a,x)\le u_{\max,\ell}
\ \text{a.e. in }Q,\ 1\le \ell\le m
\right\}.
\end{equation}
The optimal-control problem is to minimize
\[
J(y,u)
=
J_{\rm obs}(y)
+
\frac{\gamma}{2}
\int_Q |C_Iy-y_d|^2\,dx\,da\,dt
+
\frac{\alpha}{2}
\int_Q |u|^2\,dx\,da\,dt,
\qquad \alpha>0,\quad \gamma\ge0,
\]
subject to \eqref{eq:intro-state}--\eqref{eq:intro-control-set}. The default
surveillance term is time-averaged localized: with weights
$\rho_m\in L^2(0,T)$ and $\eta_m\in X$,
\[
\mathcal O_m y
=
\int_0^T\rho_m(t)\,(\eta_m,y(t))_X\,dt
=
\int_Q \big(\rho_m(t)\eta_m(a,x)\big)\cdot y(t,a,x)\,dx\,da\,dt,
\]
which is a bounded linear functional on $L^2(Q)$. Instantaneous fixed-time
observations $(\eta_m,y(t_m))_X$ and pointwise observations
$e_I\cdot y(t_m,a_m,x_m)$ are admitted only under the additional regularity of
Assumption~\ref{ass:observation-regimes}.

The main analytical contribution is the derivation of an optimality system for
a nonlocal age--space structured equation with bilinear coefficient control and
endogenous surveillance feedback. The feedback derivative generates a
separable adjoint source \(\ell_{\bar y,\bar u}(p)(t)\chi(a,x)\), which is
rank-one in the structured variables but remains time-dependent. Causality
implies that the associated feedback transfer operator is backward Volterra in
time. Under an \(L^2\)-kernel representation, this operator is
Hilbert--Schmidt and quasinilpotent \cite{gohberg1970theory,wang2026breakdown}, so the feedback-corrected adjoint is
obtained through a Volterra resolvent. This distinguishes the full
time-dependent problem from stationary or algebraic reductions, where the
feedback loop collapses to a scalar Sherman--Morrison denominator and genuine
resonance may occur.

We summarize the main results.

\paragraph{Main result A: closed-loop well-posedness.}
For every $u\in U_{\rm ad}$, the closed-loop state equation admits a unique
weak solution $y(u)\in\mathcal Y$, where
$\mathcal Y\subset L^2(0,T;V)\cap C([0,T];X)$ is the transport--diffusion
energy space of Section~\ref{sec:state-feedback-differentiability}. Time
continuity into $X$ is obtained from the characteristic mild representation;
the feedback loop is closed by a weighted Volterra contraction,
\[
\|\mathcal F(E_1)-\mathcal F(E_2)\|_\lambda
\le
\frac{C}{\lambda}
\|E_1-E_2\|_\lambda .
\]
A non-circular global a priori estimate is available under the global
one-sided dissipativity Assumption~\ref{ass:dissipativity}; without it one
obtains the same results locally in time.

\paragraph{Main result B: differentiability of the control-to-state map.}
The control-to-state map $S(u)=y(u)$ is Fr\'echet differentiable after
extension to an open $L^\infty$-neighborhood of $U_{\rm ad}$, and the
derivative $z=S'(\bar u)h$ solves the linearized age--space system, including
the feedback-derivative coefficient and boundary terms.

\paragraph{Main result C: singular adjoint and low-rank feedback correction.}
The full feedback-corrected adjoint is
$\displaystyle \mathscr A_{\rm red}^*p
=
q_{\rm run}+q_{\rm obs}
+
\ell_{\bar y,\bar u}(p)(t)\chi $,
where $\ell_{\bar y,\bar u}$ is a scalar-in-time functional and, for the
default time-averaged surveillance, $q_{\rm obs}\in L^2(Q)$. With
$\mathcal T\theta:=\ell_{\bar y,\bar u}(\mathcal G^*(\theta(t)\chi))$, the
feedback coefficient solves $(I-\mathcal T)\theta=\ell_{\bar y,\bar u}(p_{\rm
red})$. 
The feedback transfer operator is backward Volterra in time by causality.
When this operator admits an \(L^2\)-kernel representation, it is
Hilbert--Schmidt and quasinilpotent; hence \(I-\mathcal T\) is invertible and
the feedback-corrected adjoint admits an explicit Volterra-resolvent
representation. Thus genuine scalar resonance is excluded in the
time-dependent Volterra-kernel regime, whereas it may reappear after a
stationary or algebraic collapse of the feedback loop.
Furthermore,
\[
p
=
p_{\rm red}
+
\mathcal G^*
\left(
\left[
(I-\mathcal T)^{-1}
\ell_{\bar y,\bar u}(p_{\rm red})
\right](t)\chi
\right).
\]

\paragraph{Main result D: optimality system and switching correction.}
The reduced objective admits global minimizers in the localized regime
(Section~\ref{sec:optimality-switching-examples}). For a local minimizer
$\bar u$ with state $\bar y$ and feedback-corrected adjoint $\bar p$, the
gradient density $\mathcal K^*(\bar y,\bar p)$ lies in $L^1(Q)$ in general
(and in $L^2(Q)$ under the regularity of
Proposition~\ref{prop:Kstar-L2}), and
\[
\int_Q
\left(
\alpha\bar u+\mathcal K^*(\bar y,\bar p)
\right)\cdot
(u-\bar u)
\,dx\,da\,dt
\ge0
\qquad
\forall u\in U_{\rm ad},
\]
with the pointwise projection characterization
$\bar u_\ell=\Pi_{[u_{\min,\ell},u_{\max,\ell}]}(-\alpha^{-1}
\mathcal K_\ell^*(\bar y,\bar p))$ \cite{troltzsch2010optimal,yu2026beyond,hinze2008optimization}. The switching
function decomposes as $S_{\rm sw}=S_{\rm red}+S_{\rm fb}$, and a
threshold-preservation result shows that feedback changes the optimal
intervention decision only where $|S_{\rm red}|\le|S_{\rm fb}|$.

The remainder of the paper is organized as follows.
Section~\ref{sec:state-feedback-differentiability} develops the state
equation, proves closed-loop well-posedness (with time continuity from the
characteristic representation), and establishes differentiability.
Section~\ref{sec:singular-adjoint-feedback} derives the adjoint equation and
proves the unconditional operator-valued Fredholm representation.
Section~\ref{sec:optimality-switching-examples} proves existence via a conditional
transport--diffusion compactness theorem, derives the optimality system, and
analyzes the switching-function perturbation.

\section{Literature Review}
\label{sec:lite_review}
Physiologically structured population models, together with their reformulation as integral or delay equations, have been extensively investigated over the past decades. The transport (McKendrick--von Foerster) framework for age- and size-structured populations, as well as its mathematical analysis based on characteristic methods and semigroup theory, has become a classical foundation of the field~\cite{metz2014dynamics,wang2026damage,webb1985theory,liang2025global,yu2026rigorous,yu2026microscopic,gurtin1974non}. Comprehensive treatments of modelling techniques, analytical methods, and bifurcation theory can be found in standard monographs~\cite{iannelli1995mathematical,iannelli2017basic,cushing1998introduction,thieme2018mathematics,de1997gentle,metz2014dynamics,murray2007mathematical,perthame2007transport,liu2025bidirectional}. A unified formulation for nonlinear structured population models incorporating feedback through one or multiple scalar environmental variables was established in~\cite{diekmann1998formulation,diekmann2001formulation,diekmann2003steady,akimenko2018two,diekmann2020finite,diekmann2020models,yu2026microscopic,barril2022formulation,barril2024hierarchical,cai2026optimal}. This framework motivates our formulation of the closed-loop dynamics as a fixed-point problem for the environmental variable. The abstract semilinear formulation and the integrated-semigroup approach to structured population systems have been systematically developed in~\cite{magal2018theory,kang2022age,ducrot2021integrated,liang2026separation}, while local stability and bifurcation properties for representative feedback models were investigated in~\cite{diekmann2010daphnia,scarabel2021numerical,yu2026size}. Existence and stability results for transport equations with feedback, obtained through characteristic techniques and fixed-point arguments, are available in~\cite{calsina1995model,wang2026elliptic,bartlomiejczyk2015existence} and related studies.

Optimal control and harvesting problems for age- and size-structured population models have likewise received considerable attention. In particular, Pontryagin-type optimality conditions and related control theories have been extensively studied in~\cite{brokate1985pontryagin,feichtinger2003optimality,wu2025age,kumar2023stability,ni2023optimal,wu2025analysis,filho2025mathematical,kato2024measure,yu2026age}. A complementary line of research focuses on selective harvesting policies in structured populations, aiming to determine the most profitable subpopulation to harvest to maximize either discounted revenues or long-term economic returns~\cite{easterling2000size,fadlovich2025selectivity,wang2025multi,stubberud2019effects}. In contrast to these studies, the present work is primarily concerned with the well-posedness and stability analysis of the closed-loop system.

The present paper differs from this literature in three respects. First, the
control enters bilinearly as a coefficient rather than as an additive source
or a terminal harvesting functional. Second, the coefficients and renewal law
depend on an endogenous surveillance observable, so linearization produces a
feedback derivative both in the interior equation and at the renewal boundary.
Third, the adjoint feedback loop remains time-dependent, leading to a
Volterra transfer operator rather than a scalar Sherman--Morrison denominator.

\section{State equation, feedback closure, and differentiability}
\label{sec:state-feedback-differentiability}
Throughout, $T>0$, $a_{\max}>0$, and $\Omega\subset\R^d$, $d\in\{1,2,3\}$, is a
bounded Lipschitz domain. 
We set
\[
\mathcal I_a:=(0,a_{\max}),\qquad
Q:=(0,T)\times  \mathcal I_a\times\Omega ,
\]
The state is
$y=(y_1,\ldots,y_n)^\top$, and we use
\[
X:=L^2(\mathcal I_a\times\Omega;\R^n),
\qquad
V:=L^2(\mathcal I_a;H^1(\Omega;\R^n)),
\qquad
V':=L^2(\mathcal I_a;H^1(\Omega;\R^n)').
\]
We write $\langle\cdot,\cdot\rangle_{V',V}$, $(\cdot,\cdot)_X$, and
$D_{t,a}:=\partial_t+\partial_a$. The feedback observable is
\[
E_y(t)
:=
\int_{\mathcal I_a}\int_\Omega
\chi(a,x)\cdot y(t,a,x)\,dx\,da ,
\qquad 0\le t\le T,
\]
with $\chi$ fixed (regularity specified in
Assumption~\ref{ass:compactness-data}). The state equation is
\begin{equation}
\label{eq:abstract-state-equation}
D_{t,a}y
-\nabla_x\cdot(\Sigma(a,x)\nabla_x y)
+
\mathcal A(E_y(t),a,x)y
+
\mathcal N[y](t,a,x)y
+
\mathcal K(u)y
=
F
\quad\text{in }Q,
\end{equation}
with renewal boundary condition
\[
y(t,0,x)
=
\int_{\mathcal I_a}
\mathcal B(E_y(t),\alpha,x)y(t,\alpha,x)\,d\alpha ,
\qquad (t,x)\in(0,T)\times\Omega,
\]
initial condition $y(0,a,x)=y^0(a,x)$, and no-flux spatial boundary condition
\begin{equation}
\label{eq:noflux-boundary}
\Sigma(a,x)\nabla_x y(t,a,x)\cdot\nu(x)=0
\quad\text{on }(0,T)\times \mathcal I_a\times\partial\Omega .
\end{equation}
The control acts as a coefficient:
\[
\mathcal K(u)y
=
\sum_{\ell=1}^m u_\ell(t,a,x)K_\ell(a,x)y(t,a,x),
\]
with $K_\ell\in L^\infty(\mathcal I_a\times\Omega;\R^{n\times n})$, and admissible
controls
\[
U_{\rm ad}
:=
\left\{
u\in L^\infty(Q;\R^m):
u_{\min,\ell}\le u_\ell\le u_{\max,\ell}
\ \text{a.e.},\ 1\le \ell\le m
\right\}.
\]
The nonlocal operator is
\[
\mathcal N[y](t,a,x)
=
\int_{\mathcal I_a}\int_\Omega
\mathscr K(a,\alpha,x,\xi)y(t,\alpha,\xi)\,d\xi\,d\alpha ,
\]
with $\mathscr K$ a bounded kernel valued in linear maps
$\R^n\to\R^{n\times n}$, so $\mathcal N[y]y$ is an $n$-vector.

\subsection{Structural assumptions}
\label{subsec:state-assumptions}
\begin{assumption}[Diffusion]
\label{ass:diffusion}
$\Sigma: \mathcal I_a\times\Omega\to\R^{n\times n\times d\times d}$ is measurable,
bounded, and uniformly elliptic: there is $\sigma_0>0$ with
\[
\sum_{i=1}^n\sum_{j,k=1}^d
\Sigma_i^{jk}(a,x)\xi_i^j\xi_i^k
\ge
\sigma_0
\sum_{i=1}^n|\xi_i|^2
\]
for a.e. $(a,x)$ and all $\xi\in(\R^d)^n$. Moreover $\Sigma$ is independent of
$t$ (no $t$-dependence in the principal part).
\end{assumption}

\begin{assumption}[Feedback-dependent coefficients]
\label{ass:feedback-coefficients}
$\mathcal A:\R\times \mathcal I_a\times\Omega\to\R^{n\times n}$ is measurable in $(a,x)$ and
$C^2$ in the first variable, with, for every $R>0$,
\begin{equation}
\label{eq:A-bounds}
\sup_{|E|\le R}
\left(
\|\mathcal A(E,\cdot,\cdot)\|_{L^\infty}
+
\|\partial_E \mathcal A(E,\cdot,\cdot)\|_{L^\infty}
+
\|\partial_{EE}\mathcal A(E,\cdot,\cdot)\|_{L^\infty}
\right)
\le C_R .
\end{equation}
\end{assumption}

\begin{assumption}[Renewal operator]
\label{ass:renewal}
$\mathcal B:\R\times \mathcal I_a\times\Omega\to\R^{n\times n}$ is measurable in $(\alpha,x)$ and
$C^2$ in the first variable, with the analogue of \eqref{eq:A-bounds} for
$\mathcal B,\partial_E \mathcal B,\partial_{EE}\mathcal B$.
\end{assumption}

\begin{assumption}[Nonlocal interaction]
\label{ass:nonlocal}
$\mathscr K\in L^\infty(\mathcal I_a\times \mathcal I_a\times\Omega\times\Omega)$. Hence there
is $C_{\mathscr K}>0$ with, for $y_1,y_2\in X$,
\[
\begin{aligned}
\|\mathcal N[y_1]\|_{L^\infty(\mathcal I_a\times\Omega)}
&\le
C_{\mathscr K}\|y_1\|_X,
\\
\|\mathcal N[y_1]-\mathcal N[y_2]\|_{L^\infty(\mathcal I_a\times\Omega)}
&\le
C_{\mathscr K}\|y_1-y_2\|_X .
\end{aligned}
\]
Moreover $y\mapsto\mathcal N[y]$ is sequentially weak-to-strong continuous: if
$y_k\rightharpoonup y$ in $L^2(Q)$, then $\mathcal N[y_k]\to\mathcal N[y]$ in
$L^2(Q)$ (a consequence of the compactness of the kernel operator
\cite{brezis2011functional,gao2022rolling}).
\end{assumption}

\begin{assumption}[Global one-sided dissipativity]
\label{ass:dissipativity}
There is $c_0\ge0$, \emph{independent of the solution size}, such that
\begin{equation}
\label{eq:global-dissipativity}
\left[
\mathcal A(E,a,x)\eta
+
\mathcal N[y](t,a,x)\eta
+
\sum_{\ell=1}^m u_\ell K_\ell(a,x)\eta
\right]\cdot\eta
\ge
-c_0|\eta|^2
\end{equation}
for a.e. $(t,a,x)$, all $E\in\R$, all admissible $u$, all $\eta\in\R^n$, and
\emph{all} $y\in X$ (so $c_0$ does not depend on $\|y\|_X$ or $|E|$).
\end{assumption}

\begin{remark}
\label{rem:dissipativity-noncircular}
The global form \eqref{eq:global-dissipativity} is what makes the a priori
estimate of Theorem~\ref{thm:closed-loop-state-wp} non-circular: the Gronwall
rate $c_0$ does not depend on the radius of the fixed-point ball, so the ball
can be chosen as a data-dependent multiple of $\|y^0\|_X+\|F\|$ and is then
invariant. If the kernel $\mathscr K$ has a population-dissipative sign
structure, \eqref{eq:global-dissipativity} holds. If only the local bound
$c_{R}=c_0(1+R)$ is available (constant growing with $\|y\|_X\le R$), the
quadratic nonlocal term may drive finite-time blow-up; one then obtains the
same theory locally in time, with a lifespan depending on the data and the
control bounds.
\end{remark}

For the compactness and existence theory of
Section~\ref{sec:optimality-switching-examples} we use the following mild
regularity, stated here so all hypotheses appear together. It is not needed in
Sections~\ref{sec:state-feedback-differentiability}--\ref{sec:singular-adjoint-feedback}.

\begin{assumption}[Regularity of coefficients for compactness]
\label{ass:compactness-data}
The diffusion matrix satisfies \(\Sigma \in L^\infty(\Omega; \R^{d\times d})\). 
The renewal kernel \(\mathcal B = \mathcal B(E, \alpha, x)\) satisfies 
\(\mathcal B(E, \alpha, \cdot) \in W^{1,\infty}(\Omega)\) uniformly with respect to \((E, \alpha)\). 
Specifically, there exists a constant \(C_B > 0\) such that
$\|\mathcal B(E, \alpha, \cdot)\|_{W^{1,\infty}(\Omega)} \le C_B$
for all admissible \(E\) and \(\alpha \in [0, a_{\max}]\).
\end{assumption}

\subsection{Weak formulation and the solution space}
\label{subsec:weak-formulation}
We define
\begin{equation}
\label{eq:Y-space}
\mathcal Y
:=
\left\{
y\in L^2(0,T;V)\cap C([0,T];X):
D_{t,a}y\in L^2(0,T;V'),\
y(\cdot,0,\cdot)\in L^2(0,T;L^2(\Omega;\R^n))
\right\},
\end{equation}
with norm
\[
\|y\|_{\mathcal Y}
:=
\|y\|_{L^2(0,T;V)}
+
\|y\|_{C([0,T];X)}
+
\|D_{t,a}y\|_{L^2(0,T;V')}
+
\|y(\cdot,0,\cdot)\|_{L^2(0,T;L^2(\Omega))}.
\]

\begin{remark}[On time continuity]
\label{rem:time-continuity}
The membership $y\in C([0,T];X)$ in \eqref{eq:Y-space} is \emph{not} a
consequence of the standard Lions--Magenes lemma \cite{lions2012non}, which
would require $\partial_t y\in L^2(0,T;V')$. The energy estimate below controls
only the transport derivative $D_{t,a}y$. Time continuity into $X$ is instead
obtained from the characteristic mild representation
(Lemma~\ref{lem:frozen-linear-wp}): along each characteristic
$t-a=\mathrm{const}$, $y$ is given by the variation-of-constants formula for
the analytic evolution family generated by the spatial diffusion
\cite{pazy2012semigroups,acquistapace1987unified,liu2026ftu}, and parabolic smoothing yields continuity in
the evolution variable, hence in $t$. As visually distinguished in Figure~\ref{fig:transport_diffusion}, the hyperbolic age-time transport contrasts with the parabolic spatial diffusion. This clarifies the structure of the energy space $\mathcal{Y}$ and illustrates how the limitation of the standard Lions--Magenes continuity argument is bypassed.
\end{remark}

\begin{figure}[htbp]
    \centering
    \includegraphics[width=\textwidth]{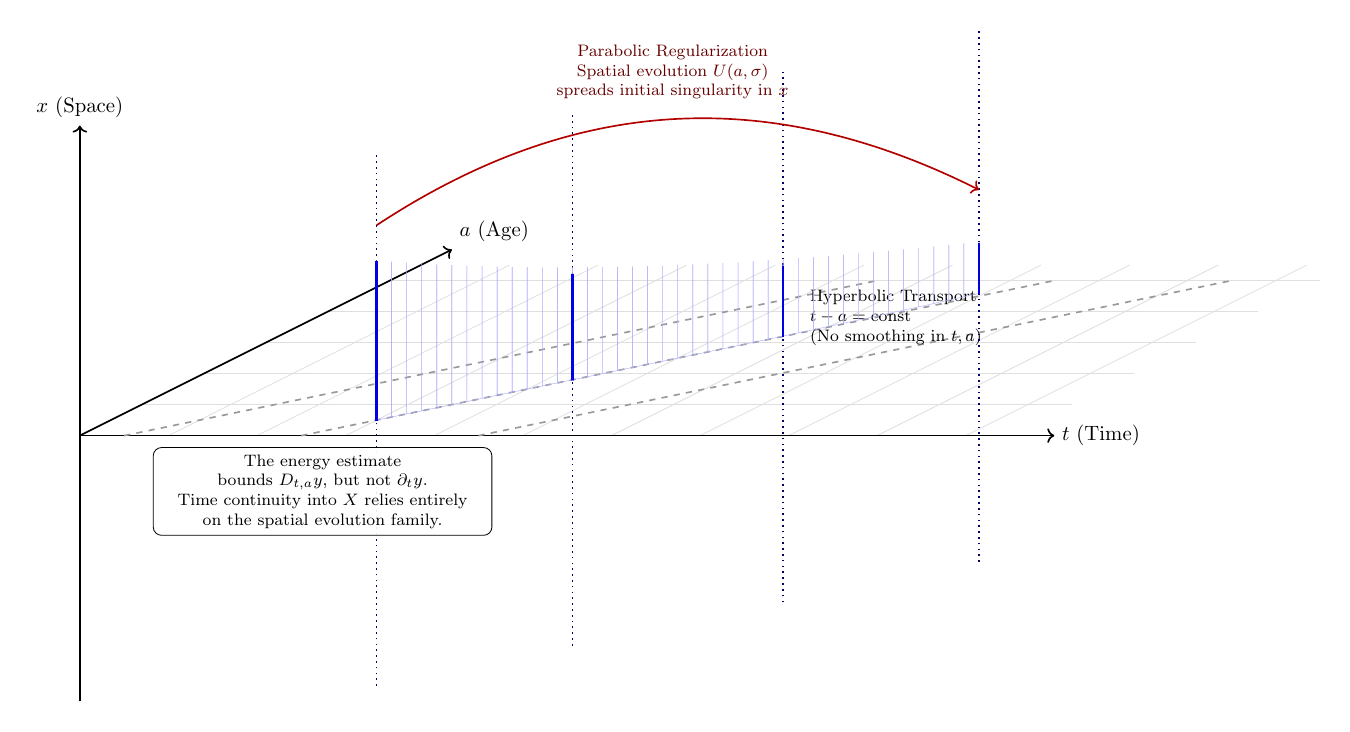} 
    \caption{Characteristic transport vs. spatial diffusion. The figure contrasts the hyperbolic age-time transport along characteristics ($t-a=\mathrm{const}$), which exhibits no smoothing, with the parabolic spatial diffusion driven by the evolution family $U(a,\sigma)$ in the $x$-direction. This structural difference clarifies the nature of the energy space $\mathcal{Y}$ and explains the bypass of the standard Lions--Magenes continuity argument.}
    \label{fig:transport_diffusion}
\end{figure}

For $E\in C([0,T])$, $w\in C([0,T];X)$, $u\in L^\infty(Q;\R^m)$, set the frozen
coefficient
\[
C_{E,w,u}(t,a,x)
:=
\mathcal A(E(t),a,x)
+
\mathcal N[w](t,a,x)
+
\sum_{\ell=1}^m u_\ell K_\ell(a,x),
\]
and the frozen renewal operator
$\displaystyle \mathcal R_E(t)\varphi(x)
:=
\int_{\mathcal I_a}
\mathcal B(E(t),\alpha,x)\varphi(\alpha,x)\,d\alpha $,
which by Assumption~\ref{ass:renewal} satisfies
\begin{equation}
\label{eq:renewal-estimate}
\|\mathcal R_E(t)\varphi\|_{L^2(\Omega)}
\le
C_R a_{\max}^{1/2}\|\varphi\|_X
\quad
\forall |E(t)|\le R,\ \varphi\in X .
\end{equation}
A weak solution $y\in\mathcal Y$ of
\eqref{eq:abstract-state-equation}--\eqref{eq:noflux-boundary} satisfies
$y(0)=y^0$, the renewal condition
$y(t,0,\cdot)=\mathcal R_{E_y}(t)y(t,\cdot,\cdot)$ for a.e.\ $t$, and
\begin{multline*}
\langle D_{t,a}y(t),\varphi\rangle_{V',V}
+
\int_{\mathcal I_a}\int_\Omega
\Sigma\nabla_x y:\nabla_x\varphi\,dx\,da
\\
+
\int_{\mathcal I_a}\int_\Omega
\left[
\mathcal A(E_y(t))y
+
\mathcal N[y]y
+
\mathcal K(u)y
\right]\cdot\varphi\,dx\,da
=
\langle F(t),\varphi\rangle_{V',V}
\end{multline*}
for every $\varphi\in V$ and a.e.\ $t$.

\subsection{The frozen linear problem}
\label{subsec:frozen-linear-problem}
Given $E\in C([0,T])$, $w\in C([0,T];X)$, $u\in L^\infty(Q;\R^m)$,
$G\in L^2(0,T;V')$, $y^0\in X$, consider
\begin{equation}
\label{eq:frozen-linear-equation}
D_{t,a}y
-\nabla_x\cdot(\Sigma\nabla_x y)
+
C_{E,w,u}y
=
G,
\qquad
y(t,0,\cdot)=\mathcal R_E(t)y(t,\cdot,\cdot),
\quad
y(0)=y^0 .
\end{equation}
Let $\{U(a,\sigma)\}_{0\le\sigma\le a\le a_{\max}}$ denote the parabolic
evolution family on $X_\Omega:=L^2(\Omega;\R^n)$ generated by the family
$\{-\nabla_x\cdot(\Sigma(a,\cdot)\nabla_x\cdot)\}_a$ with the no-flux boundary
condition \cite{acquistapace1987unified,pazy2012semigroups,liu2026computational}. Since $\Omega$ is bounded, each
generator has compact resolvent, so $U(a,\sigma)$ is a \emph{compact} operator
on $X_\Omega$ for $a>\sigma$, is uniformly bounded, and satisfies the analytic
smoothing estimate
$\|U(a,\sigma)\|_{\mathcal L((H^1)',L^2)}\le C(a-\sigma)^{-1/2}$.

\begin{lemma}[Frozen transport--diffusion--renewal problem]
\label{lem:frozen-linear-wp}
Let Assumptions~\ref{ass:diffusion}--\ref{ass:dissipativity} hold.
Let $E\in C([0,T])$, $w\in C([0,T];X)$, $u\in L^\infty(Q;\R^m)$,
$G\in L^2(0,T;V')$, $y^0\in X$. Then \eqref{eq:frozen-linear-equation} has a
unique weak solution $y\in\mathcal Y$, with $y\in C([0,T];X)$, and there is a
constant $C$ depending on the data and on $c_0$, $\|E\|_{C}$, $\|w\|_C$,
$\|u\|_\infty$ but \emph{not growing with the solution}, such that
\begin{equation}
\label{eq:frozen-estimate}
\|y\|_{\mathcal Y}
\le
C
\left(
\|y^0\|_X
+
\|G\|_{L^2(0,T;V')}
\right).
\end{equation}
\end{lemma}

\begin{proof}
\emph{Energy estimate.} Testing \eqref{eq:frozen-linear-equation} by $y$ over
$ \mathcal I_a\times \Omega$ gives
\begin{multline*}
\frac12\frac{d}{dt}\|y(t)\|_X^2
+
\frac12\|y(t,a_{\max},\cdot)\|_{L^2(\Omega)}^2
-
\frac12\|y(t,0,\cdot)\|_{L^2(\Omega)}^2
+
\sigma_0\|\nabla_x y(t)\|_{L^2(\mathcal I_a\times\Omega)}^2
\\
\le
c_0\|y(t)\|_X^2
+
\|G(t)\|_{V'}\|y(t)\|_V ,
\end{multline*}
using global dissipativity \eqref{eq:global-dissipativity} for the
zeroth-order terms; the outflow term
$\tfrac12\|y(t,a_{\max},\cdot)\|^2_{L^2(\Omega)}$ is retained and yields, after
integration, a bound on the outflow trace in $L^2((0,T)\times\Omega)$. With
\eqref{eq:renewal-estimate}, $\|y(t,0,\cdot)\|^2_{L^2(\Omega)}\le
C_R^2a_{\max}\|y(t)\|_X^2$, so after Young's inequality and Gronwall (rate
$c_0$, independent of the solution size),
$\displaystyle \|y\|_{C([0,T];X)}+\|y\|_{L^2(0,T;V)}
\le
C(\|y^0\|_X+\|G\|_{L^2(0,T;V')})$,
and the equation gives the same bound on $\|D_{t,a}y\|_{L^2(V')}$, while
\eqref{eq:renewal-estimate} gives the trace at $a=0$.

\emph{Time continuity and construction.} Along characteristics $t-a=s$, set
$Y(s,a,x):=y(s+a,a,x)$; then $\partial_a Y-\nabla_x\cdot(\Sigma\nabla_x Y)=
(G-C_{E,w,u}y)(s+a,a,\cdot)$, a parabolic Cauchy problem in $(a;x)$ with
inflow given by $y^0$ (for $s<0$) or by the renewal trace (for $s>0$). The
variation-of-constants formula
\[
Y(s,a,\cdot)
=
U(a,a_0)\,Y(s,a_0,\cdot)
+
\int_{a_0}^a U(a,\sigma)\,(G-C_{E,w,u}y)(s+\sigma,\sigma,\cdot)\,d\sigma
\]
defines a contraction in $C_a(X_\Omega)$ on small age-intervals, since the
analytic smoothing $\|U(a,\sigma)\|_{(H^1)'\to L^2}\le C(a-\sigma)^{-1/2}$ is
integrable; iterating over $\mathcal I_a$ and closing the renewal map by
\eqref{eq:renewal-estimate} produces a unique mild solution that is continuous
in $a$ valued in $X_\Omega$, hence $y\in C([0,T];X)$. This mild solution
coincides with the weak solution and obeys \eqref{eq:frozen-estimate}.
Uniqueness follows from the energy estimate applied to the difference of two
solutions with zero data.
\end{proof}

\begin{lemma}[Stability of the frozen solution]
\label{lem:frozen-stability}
Let Assumptions~\ref{ass:diffusion}--\ref{ass:dissipativity} hold.
Let $y_i$ solve \eqref{eq:frozen-linear-equation} for
$(E_i,w_i,u,G_i,y_i^0)$, with
$\|E_i\|_C+\|w_i\|_C+\|u\|_\infty\le R$. Then
\begin{multline}
\label{eq:frozen-stability-estimate}
\|y_1-y_2\|_{C([0,t];X)}^2
+
\|y_1-y_2\|_{L^2(0,t;V)}^2
\\
\le
C_R
\left[
\|y_1^0-y_2^0\|_X^2
+
\|G_1-G_2\|_{L^2(0,t;V')}^2
+
\int_0^t
\left(
|E_1-E_2|^2
+
\|w_1-w_2\|_X^2
\right)\,ds
\right]
\end{multline}
for all $t\in[0,T]$.
\end{lemma}

\begin{proof}
The difference $\zeta=y_1-y_2$ solves a frozen equation with right-hand side
$(C_{E_2,w_2,u}-C_{E_1,w_1,u})y_2$ and renewal defect
$(\mathcal R_{E_1}-\mathcal R_{E_2})y_2$. By the Lipschitz bounds of
Assumptions~\ref{ass:feedback-coefficients}--\ref{ass:nonlocal} and
\eqref{eq:frozen-estimate},
$\|(C_{E_2,w_2,u}-C_{E_1,w_1,u})y_2\|_{V'}\le C_R(|E_1-E_2|+\|w_1-w_2\|_X)$ and
$\|(\mathcal R_{E_1}-\mathcal R_{E_2})y_2\|_{L^2(\Omega)}\le C_R|E_1-E_2|$. The
energy estimate of Lemma~\ref{lem:frozen-linear-wp} gives
\eqref{eq:frozen-stability-estimate}.
\end{proof}

\subsection{Closed-loop feedback well-posedness}
\label{subsec:closed-loop-wp}
For fixed $u$, let $\mathscr T_u(E,w)=y$ solve
\eqref{eq:frozen-linear-equation} with $G=F$. The closed-loop system is the
fixed point $y=\mathscr T_u(E,y)$, $E=E_y$. For $\lambda>0$ set
$\|E\|_\lambda:=\sup_t e^{-\lambda t}|E(t)|$,
$\|y\|_{\lambda,X}:=\sup_t e^{-\lambda t}\|y(t)\|_X$.

\begin{lemma}[Weighted Volterra estimate]
\label{lem:feedback-volterra}
Let $u\in U_{\rm ad}$, and let $y_i=\mathscr T_u(E_i,y_i)$ with
$\|E_i\|_C+\|y_i\|_C\le R$. Then with
$\mathcal F(E_i)(t):=\int_{\mathcal I_a}\int_\Omega\chi\cdot y_i$, there is $C_R>0$,
independent of $\lambda$, with
$\displaystyle \|\mathcal F(E_1)-\mathcal F(E_2)\|_\lambda
\le
\frac{C_R}{\lambda}
\|E_1-E_2\|_\lambda$
for $\lambda$ large.
\end{lemma}

\begin{proof}
With $\zeta=y_1-y_2$, the energy estimate gives
$\|\zeta(t)\|_X\le C_R\int_0^t(\|\zeta(s)\|_X+|E_1(s)-E_2(s)|)\,ds$.
Multiplying by $e^{-\lambda t}$ and using
$e^{-\lambda t}\int_0^t e^{\lambda s}ds\le\lambda^{-1}$ yields
$\|\zeta\|_{\lambda,X}\le\frac{C_R}{\lambda}(\|\zeta\|_{\lambda,X}+\|E_1-E_2\|_\lambda)$;
for $\lambda>2C_R$, $\|\zeta\|_{\lambda,X}\le\frac{C_R}{\lambda}\|E_1-E_2\|_\lambda$.
Since $|\mathcal F(E_1)-\mathcal F(E_2)|\le\|\chi\|_X\|\zeta\|_X$, the claim
follows.
\end{proof}

\begin{theorem}[Closed-loop state well-posedness]
\label{thm:closed-loop-state-wp}
Let Assumptions~\ref{ass:diffusion}--\ref{ass:dissipativity} hold,
$F\in L^2(0,T;V')$, $y^0\in X$, $u\in U_{\rm ad}$. Then
\eqref{eq:abstract-state-equation}--\eqref{eq:noflux-boundary} has a unique
weak solution $y(u)\in\mathcal Y$, $E_u\in C([0,T])$, and there is $C$ depending
on the data and the $L^\infty$-bounds of $U_{\rm ad}$ but not on $u$ such that
\begin{equation}
\label{eq:closed-loop-apriori}
\|y(u)\|_{\mathcal Y}
+
\|E_u\|_{C([0,T])}
\le
C
\left(
\|y^0\|_X+\|F\|_{L^2(0,T;V')}
\right).
\end{equation}
\end{theorem}

\begin{proof}
By Lemma~\ref{lem:frozen-linear-wp}, $\mathscr T_u$ is well defined, and by
\eqref{eq:frozen-estimate} with the $R$-independent constant (global
dissipativity), the data-bound $R_0:=2C(\|y^0\|_X+\|F\|)$ makes the ball
$\mathbb B_{R_0}=\{(E,w):\|E\|_C+\|w\|_C\le R_0\}$ invariant under
$\mathscr M_u(E,w):=(E_{\mathscr T_u(E,w)},\mathscr T_u(E,w))$. 
By Lemma~\ref{lem:frozen-stability} and the weighted-norm computation of
Lemma~\ref{lem:feedback-volterra}, $\mathscr M_u$ is a contraction in
$\|\cdot\|_\lambda$ for $\lambda$ large, giving a unique fixed point
$(E_u,y_u)$, which solves the closed-loop system. The bound
\eqref{eq:closed-loop-apriori} is \eqref{eq:frozen-estimate} at the fixed
point. Uniqueness follows from a Gronwall argument on the difference of two
closed-loop solutions.
\end{proof}

As illustrated in Figure~\ref{fig:feedback_architecture}, the closed-loop feedback architecture maps the nonlinear control-to-state relations, highlighting the endogenous feedback closure. This block diagram explicitly tracks the exact dependencies of the fixed-point mapping $y = \mathscr{T}_u(E, y)$, which is crucial to justify the weighted Volterra contraction estimate in the $\|\cdot\|_\lambda$ norm.

\begin{figure}[htbp]
    \centering
    \includegraphics[width=\textwidth]{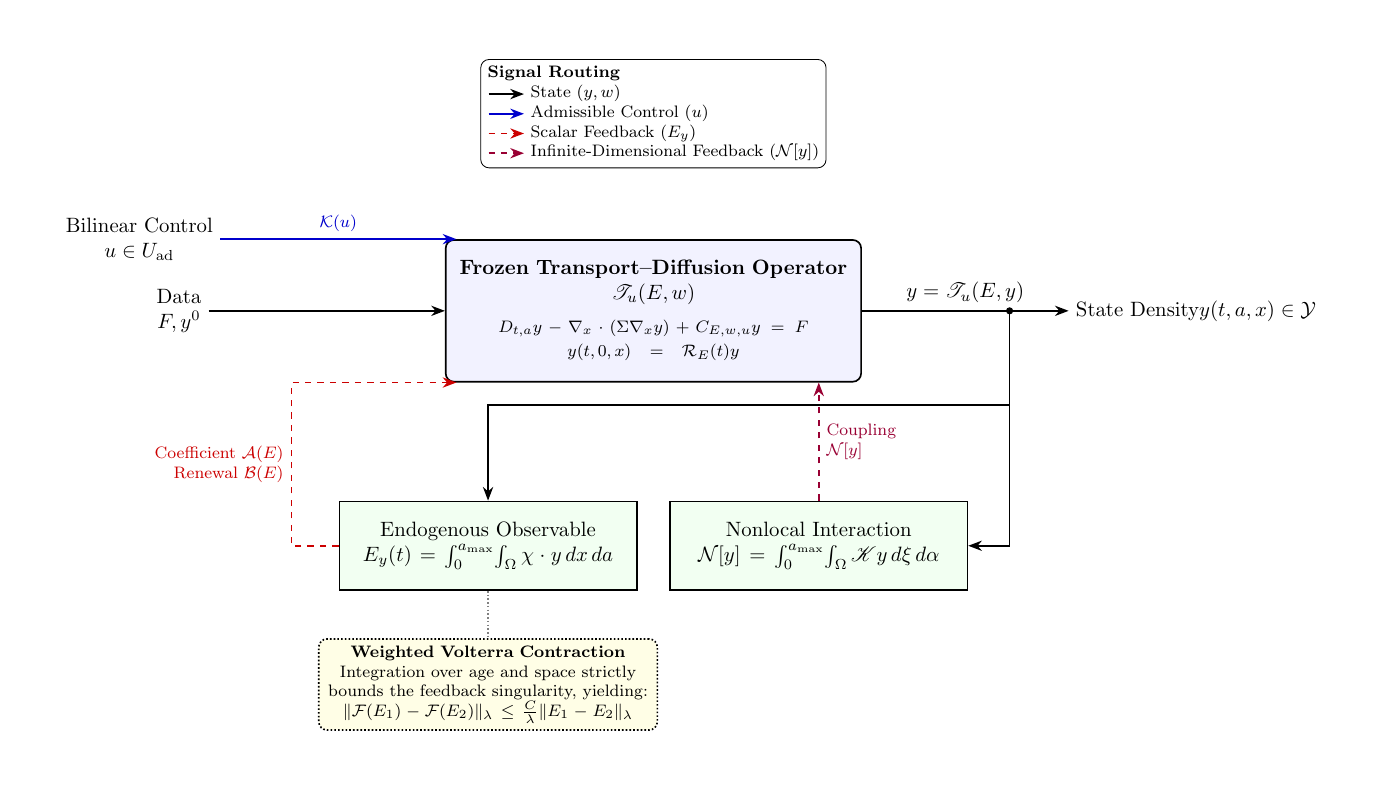} 
    \caption{The closed-loop feedback architecture. This operator block diagram maps the nonlinear control-to-state map and highlights the endogenous feedback closure. Distinct pathways illustrate the state variables ($y$), the endogenous observable ($E_y(t)$), and the bilinear control ($u$). Tracking these exact dependencies of the fixed-point mapping $y = \mathscr{T}_u(E, y)$ justifies the weighted Volterra contraction estimate via the delay/integral nature of the observable.}
    \label{fig:feedback_architecture}
\end{figure}

\subsection{Control-to-state differentiability}
\label{subsec:control-to-state-differentiability}
Extend the state equation to an open $L^\infty$-neighborhood
$\mathcal U\supset U_{\rm ad}$ with a uniform bound, so
Theorem~\ref{thm:closed-loop-state-wp} applies on $\mathcal U$. Fix
$\bar u\in\mathcal U$, $\bar y=S(\bar u)$, $\bar E=E_{\bar y}$. For
$h\in L^\infty(Q;\R^m)$ the derivative $z=S'(\bar u)h$ solves
\begin{equation}
\label{eq:linearized-state}
D_{t,a}z
-\nabla_x\cdot(\Sigma\nabla_x z)
+
\mathcal A(\bar E)z
+
\mathcal N[\bar y]z
+
\mathcal N[z]\bar y
+
\mathcal K(\bar u)z
+
\mathcal K(h)\bar y
+
\partial_E \mathcal A(\bar E)\delta E(t)\bar y
=
0
\end{equation}
with $\delta E(t)=\int_{\mathcal I_a}\int_\Omega\chi\cdot z$, linearized renewal
\begin{equation}
\label{eq:linearized-renewal}
z(t,0,x)
=
\int_{\mathcal I_a}
\mathcal B(\bar E)z\,d\alpha
+
\delta E(t)
\int_{\mathcal I_a}
\partial_E \mathcal B(\bar E)\bar y\,d\alpha ,
\end{equation}
and $z(0)=0$.

\begin{lemma}[Well-posedness of the linearized system]
\label{lem:linearized-wp}
For every $h\in L^\infty(Q;\R^m)$ the system
\eqref{eq:linearized-state}--\eqref{eq:linearized-renewal} has a unique
solution $z\in\mathcal Y$, with
$\|z\|_{\mathcal Y}+\|\delta E\|_{C([0,T])}\le C\|h\|_{L^\infty(Q)}$.
\end{lemma}

\begin{proof}
The system has the transport--diffusion--renewal structure of the frozen
problem, with extra lower-order terms $\mathcal N[z]\bar y$,
$\partial_E \mathcal A(\bar E)\delta E\,\bar y$ obeying
$\|\mathcal N[z]\bar y\|_{V'}\le C\|z\|_X$,
$\|\partial_E \mathcal A(\bar E)\delta E\,\bar y\|_{V'}\le C|\delta E|$,
$|\delta E|\le\|\chi\|_X\|z\|_X$, and a feedback renewal term bounded by
$C|\delta E|$. The energy estimate of Lemma~\ref{lem:frozen-linear-wp} and
Gronwall give the claim.
\end{proof}

\begin{theorem}[Fr\'echet differentiability]
\label{thm:control-to-state-differentiability}
Under Assumptions~\ref{ass:diffusion}--\ref{ass:dissipativity} (with $\mathcal A, \mathcal B$
$C^2$ in $E$), $S:\mathcal U\to\mathcal Y$ is Fr\'echet differentiable, with
$S'(\bar u)h=z$ the solution of
\eqref{eq:linearized-state}--\eqref{eq:linearized-renewal}, and
$\|S'(\bar u)h\|_{\mathcal Y}\le C\|h\|_{L^\infty(Q)}$.
\end{theorem}

\begin{proof}
Let $z_\varepsilon=(y_\varepsilon-\bar y)/\varepsilon$,
$y_\varepsilon=S(\bar u+\varepsilon h)$. Expanding the coefficient, renewal,
nonlocal, and control increments as in the standard difference-quotient
computation, the remainders vanish in the energy norms by the uniform
Lipschitz stability of the state equation, and $z_\varepsilon-z$ solves the
linearized system with vanishing data; Lemma~\ref{lem:linearized-wp} gives
$z_\varepsilon\to z$ in $\mathcal Y$. Replacing $\varepsilon h$ by a general
perturbation upgrades G\^ateaux to Fr\'echet differentiability.
\end{proof}

\begin{proposition}[$C^1$-regularity]
\label{prop:C1-state-map}
$S:\mathcal U\to\mathcal Y$ is of class $C^1$.
\end{proposition}

\begin{proof}
If $u_k\to u$ in $L^\infty$, the nonlinear stability estimate gives
$\|S(u_k)-S(u)\|_{\mathcal Y}+\|E_{y_k}-E_y\|_C\to0$; subtracting linearized
equations and using these convergences yields
$\|S'(u_k)-S'(u)\|_{\mathcal L(L^\infty,\mathcal Y)}\to0$.
\end{proof}

\section{Singular adjoints and low-rank feedback perturbations}
\label{sec:singular-adjoint-feedback}
Fix $\bar u\in U_{\rm ad}$, $\bar y=S(\bar u)$, $\bar E=E_{\bar y}$. Let
\[
\mathscr A_{\rm red}z
:=
D_{t,a}z
-\nabla_x\cdot(\Sigma\nabla_x z)
+
\mathcal A(\bar E)z
+
\mathcal N[\bar y]z
+
\mathcal N[z]\bar y
+
\mathcal K(\bar u)z ,
\]
\[
\mathscr B_{\rm red}z
:=
z(t,0,x)
-
\int_{\mathcal I_a}
\mathcal B(\bar E)z\,d\alpha .
\]
The full linearized equation reads $\mathscr A_{\rm red}z+\delta E_z(t)
c_{\bar y}+\mathcal K(h)\bar y=0$, $\mathscr B_{\rm red}z=\delta E_z(t)
b_{\bar y}$, where
\[
\delta E_z(t):=\int_{\mathcal I_a}\int_\Omega\chi\cdot z,
\quad
c_{\bar y}:=\partial_E\mathcal A(\bar E)\bar y,
\quad
b_{\bar y}(t,x):=\int_{\mathcal I_a}\partial_E\mathcal B(\bar E)\bar y\,d\alpha .
\]

\subsection{Observation models}
\label{subsec:observation-models}
\begin{assumption}[Observation regimes]
\label{ass:observation-regimes}
The analysis is carried out in one of the following regimes.

\emph{(O1) Time-averaged localized (default).} With weights
$\rho_m\in L^2(0,T)$ and $\eta_m\in X$,
\[
\mathcal O_m y
=
\int_0^T\rho_m(t)\,(\eta_m,y(t))_X\,dt
=
(\Theta_m,y)_{L^2(Q)},
\qquad
\Theta_m(t,a,x):=\rho_m(t)\eta_m(a,x)\in L^2(Q).
\]
Each $\mathcal O_m$ is bounded on $L^2(Q)$, hence continuous under weak
$L^2(Q)$ convergence.

\emph{(O2) Instantaneous localized (conditional).} $\mathcal O_m
y=(\eta_m,y(t_m))_X$ at fixed $t_m$. Admitted only when one has compactness in
$C([0,T];X)$, so that $y_k(t_m)\to y(t_m)$ in $X$; this requires control of
$\partial_t y_k$ in a negative space, which the energy theory does not supply,
and so must be furnished separately.

\emph{(O3) Pointwise (conditional, strong regime).} $S(\mathcal U)\subset
\mathcal Y_\sharp\hookrightarrow C(\overline Q;\R^n)$ and the reduced adjoint
maps point sources into $\mathcal P_\sharp\hookrightarrow L^2(Q;\R^n)$ with
$\mathcal K^*(\bar y,\bar p)\in L^2(Q)$; then
$\mathcal O_m y=e_I\cdot y(t_m,a_m,x_m)$.
\end{assumption}

\paragraph{Default (O1) source.}
With surveillance cost
$J_{\rm obs}(y)=\tfrac12\sum_{m=1}^M|\mathcal O_m y-d_m|^2$ and residuals
$r_m=\mathcal O_m\bar y-d_m$, the derivative is
$\displaystyle J_{\rm obs}'(\bar y)z
=
\sum_{m=1}^M r_m(\Theta_m,z)_{L^2(Q)}$,
and the observation source is the genuine $L^2(Q)$ function
$\displaystyle q_{\rm obs}
=
\sum_{m=1}^M r_m\,\Theta_m
=
\sum_{m=1}^M r_m\,\rho_m(t)\eta_m(a,x)\in L^2(Q)$.
No time-measure appears; the adjoint source is non-singular.

\begin{remark}[Conditional (O3) sufficient condition]
\label{rem:point-sufficient}
Regime (O3) holds, for instance, under age regularization
$\partial_t y+\partial_a y-\varepsilon\partial_{aa}y-\nabla_x\cdot(\Sigma
\nabla_x y)$, $\varepsilon>0$, which is uniformly parabolic and second order in
the \emph{joint} variable $(a,x)\in\R^{d+1}$. If
$y\in W^{1,p}(0,T;L^p(\mathcal I_a\times\Omega))\cap  L^p(0,T;W^{2,p}(\mathcal I_a\times\Omega))$
($W^{2,p}$ jointly in $(a,x)$) for some $p>(d+3)/2$, then
$y\in C^\beta(\overline Q)$ with $\beta=2-(d+3)/p>0$, since the spatial-type
dimension is $N=d+1$ and the threshold is $(N+2)/2=(d+3)/2$ (for $d=1,2,3$,
$p>2,\tfrac52,3$). This is a different PDE from the bare hyperbolic model; we
therefore present (O3) as a conditional extension, not on a par with (O1).
\end{remark}

\subsection{Reduced adjoint equation}
\label{subsec:reduced-adjoint}
Let $q_{\rm run}=\gamma C_I^\top(C_I\bar y-y_d)\in L^2(Q)$ be the running-cost
source and $q:=q_{\rm run}+q_{\rm obs}$. The reduced adjoint is, formally,
\begin{multline}
\label{eq:formal-reduced-adjoint}
-\partial_t p-\partial_a p
-\nabla_x\cdot(\Sigma^\top\nabla_x p)
+
\mathcal A(\bar E)^\top p
+
\mathcal N[\bar y]^\top p
+
\mathcal N_{\bar y}^*p
+
\mathcal K(\bar u)^\top p
-
\mathcal R_{\bar E}^*[p(t,0,\cdot)]
=
q
\end{multline}
with $p(T,a,x)=0$, $p(t,a_{\max},x)=0$, and no-flux in $x$; here
$\mathcal R_{\bar E}^*[\zeta](t,a,x)=\mathcal B(\bar E)^\top\zeta(t,x)$ with
$\zeta=p(t,0,\cdot)$, and $\mathcal N_{\bar y}^*$ is the adjoint of
$z\mapsto\mathcal N[z]\bar y$. Equivalently, by transposition, $p_{\rm red}$
satisfies $\int_0^T\langle f,p_{\rm red}\rangle\,dt=\langle q,z_f\rangle$ for
the reduced forward solutions $z_f$ of $\mathscr A_{\rm red}z_f=f$,
$\mathscr B_{\rm red}z_f=0$, $z_f(0)=0$.

\begin{proposition}[Reduced adjoint solvability: default regime]
\label{prop:reduced-adjoint-solvability}
In regime (O1), $q=q_{\rm run}+q_{\rm obs}\in L^2(Q)\hookrightarrow
L^2(0,T;V')$, and the backward problem \eqref{eq:formal-reduced-adjoint} has a
unique solution $p_{\rm red}$ in the backward energy space
$\mathcal P:=\{p\in L^2(0,T;V)\cap  C([0,T];X):
D_{t,a}p\in L^2(0,T;V')\}$, with
\begin{equation}
\label{eq:Gstar-bound}
\|p_{\rm red}\|_{\mathcal P}=\|\mathcal G^* q\|_{\mathcal P}
\le
C\|q\|_{L^2(Q)} .
\end{equation}
\end{proposition}

\begin{proof}
Equation \eqref{eq:formal-reduced-adjoint} is a backward
transport--diffusion--renewal problem with $L^2(0,T;V')$ source; reversing time
$\tau=T-t$ and age $b=a_{\max}-a$ turns it into a forward problem of the type
covered by Lemma~\ref{lem:frozen-linear-wp} (the nonlocal adjoint
$\mathcal N_{\bar y}^*$ and the renewal adjoint are bounded lower-order and
boundary terms). Existence, uniqueness, time continuity, and
\eqref{eq:Gstar-bound} follow.
\end{proof}

In regimes (O2)/(O3) the source $q_{\rm obs}$ is more singular and
Proposition~\ref{prop:reduced-adjoint-solvability} is replaced by a hypothesis:
there is a Banach space $\mathcal P$ continuously embedded in the dual of the
source space on which $\mathcal G^*$ is bounded. For (O3) this is supplied by
the smoothing regime of Remark~\ref{rem:point-sufficient}.

\subsection{Feedback derivative and the low-rank adjoint term}
\label{subsec:feedback-adjoint-term}
Define
\begin{equation}
\label{eq:ell-definition}
\ell_{\bar y,\bar u}(p)(t)
:=
-
\int_{\mathcal I_a}\int_\Omega
c_{\bar y}\cdot p\,dx\,da
+
\int_\Omega
b_{\bar y}\cdot p(t,0,x)\,dx .
\end{equation}
The duality identity
\[
\langle\delta E_z c_{\bar y},p\rangle
-
\int_0^T\!\!\int_\Omega\delta E_z(t)b_{\bar y}\cdot p(t,0,x)\,dx\,dt
=
-\int_0^T\ell_{\bar y,\bar u}(p)(t)\,\delta E_z(t)\,dt
\]
and $\delta E_z(t)=\int_{\mathcal I_a}\int_\Omega\chi\cdot z$ give
$\int_0^T\ell_{\bar y,\bar u}(p)\delta E_z\,dt=\int_Q\ell_{\bar y,\bar u}(p)(t)
\chi\cdot z$, so the feedback forcing on the adjoint is the rank-one (in
age--space) term $\ell_{\bar y,\bar u}(p)(t)\chi(a,x)$. The full
feedback-corrected adjoint is
\begin{equation}
\label{eq:feedback-corrected-adjoint-abstract}
\mathscr A_{\rm red}^*p
=
q_{\rm run}+q_{\rm obs}
+
\ell_{\bar y,\bar u}(p)(t)\chi .
\end{equation}

\subsection{Causality and unconditional solvability}
\label{subsec:causality-quasinilpotent}
Let $p_{\rm red}=\mathcal G^*(q_{\rm run}+q_{\rm obs})$. If $p$ solves
\eqref{eq:feedback-corrected-adjoint-abstract} then, by linearity,
\begin{equation}
\label{eq:p-corrected-preliminary}
p=p_{\rm red}+\mathcal G^*(\ell_{\bar y,\bar u}(p)(t)\chi).
\end{equation}
With $\theta:=\ell_{\bar y,\bar u}(p)$ and the feedback transfer operator
\[
\mathcal T:L^2(0,T)\to L^2(0,T),
\qquad
\mathcal T\theta:=\ell_{\bar y,\bar u}(\mathcal G^*(\theta(t)\chi)),
\]
applying $\ell_{\bar y,\bar u}$ gives
\begin{equation}
\label{eq:I-minus-T}
(I-\mathcal T)\theta=\ell_{\bar y,\bar u}(p_{\rm red}).
\end{equation}
For $\tau\in[0,T]$ let $P_\tau,P_\tau^{\rm c}$ be multiplication by
$\mathbf1_{(\tau,T]},\mathbf1_{[0,\tau]}$, and $M_\chi\theta:=\theta(t)\chi$,
so $\mathcal T=\ell_{\bar y,\bar u}\circ\mathcal G^*\circ M_\chi$.

\begin{lemma}[Anti-causality]
\label{lem:anticausality}
The reduced forward operator $\mathcal G$ is causal,
$P_\tau^{\rm c}\mathcal G=P_\tau^{\rm c}\mathcal G P_\tau^{\rm c}$; its
transpose is anti-causal,
$P_\tau\mathcal G^*=P_\tau\mathcal G^*P_\tau$; and $\mathcal T$ is backward
Volterra:
\begin{equation}
\label{eq:T-anticausality}
P_\tau\,\mathcal T=P_\tau\,\mathcal T\,P_\tau,
\qquad\tau\in[0,T].
\end{equation}
\end{lemma}

\begin{proof}
Causality: in $\mathscr A_{\rm red}z=f$ every operator except $\partial_t$ is
local in time, so the energy identity on $[0,\tau]$ shows $z_f|_{[0,\tau]}$
depends only on $f|_{[0,\tau]}$. Anti-causality follows by taking
time-pairing adjoints and using $(P_\tau^{\rm c})^*=P_\tau^{\rm c}$. Finally
$M_\chi$ and $\ell_{\bar y,\bar u}$ are local in time, so
$P_\tau\mathcal T=\ell_{\bar y,\bar u}P_\tau\mathcal G^*M_\chi
=\ell_{\bar y,\bar u}P_\tau\mathcal G^*P_\tau M_\chi=P_\tau\mathcal T P_\tau$.
\end{proof}

By \eqref{eq:T-anticausality}, if $\mathcal T$ is an integral operator its
kernel is backward-triangular. We adopt the kernel representation as a
hypothesis; Proposition~\ref{prop:kernel-interior} below records the rigorous
content available from the energy theory.

\begin{hypothesis}[Volterra kernel representation]
\label{hyp:volterra-kernel}
$\mathcal T$ admits the representation
\begin{equation}
\label{eq:T-kernel}
(\mathcal T\theta)(t)=\int_t^T\kappa(t,s)\,\theta(s)\,ds,
\qquad
\kappa\in L^2(\triangle_T),\quad
\triangle_T:=\{0\le t\le s\le T\},
\end{equation}
with $\kappa(t,s)=0$ for $s<t$.
\end{hypothesis}

\begin{proposition}[Kernel structure from the energy theory]
\label{prop:kernel-interior}
In regime (O1), $\mathcal T$ is a bounded operator on $L^2(0,T)$, and by
Lemma~\ref{lem:anticausality} it is backward-triangular: $P_\tau\mathcal T=
P_\tau\mathcal T P_\tau$. Write $\mathcal T=\mathcal T_{\rm int}+
\mathcal T_{\rm tr}$ according to the two terms of $\ell_{\bar y,\bar u}$ in
\eqref{eq:ell-definition}. Let $G(t,s)\in\mathcal L(X)$ denote the bounded
reduced-adjoint propagator, so that
$\mathcal G^*(\theta\chi)(t)=\int_t^T G(t,s)\,\theta(s)\chi\,ds$ with
$\sup_{t\le s}\|G(t,s)\|_{\mathcal L(X)}\le C$
\textnormal{(Proposition~\ref{prop:reduced-adjoint-solvability})}. Then the
interior part is an integral operator with kernel
\begin{equation}
\label{eq:kappa-int}
\kappa_{\rm int}(t,s)=-\big(c_{\bar y}(t),G(t,s)\chi\big)_X,
\qquad
|\kappa_{\rm int}(t,s)|\le C\|\chi\|_X\,\|c_{\bar y}(t)\|_X ,
\end{equation}
and since $c_{\bar y}\in L^2(0,T;X)$ \textnormal{(as
$\bar y\in C([0,T];X)$ and $\partial_E\mathcal A$ is bounded)},
$\kappa_{\rm int}\in L^2(\triangle_T)$ with no diagonal singular part. The
kernel has no Dirac mass on $\{t=s\}$ because the source enters
\eqref{eq:formal-reduced-adjoint} through $\partial_t$, so $G(t,s)\to I$ as
$t\uparrow s$ with bounded diagonal values. Hence
Hypothesis~\ref{hyp:volterra-kernel} reduces to the corresponding
$L^2(\triangle_T)$ statement for the renewal-trace part $\kappa_{\rm tr}$,
which holds whenever the propagator renewal trace
$s\mapsto(G(\cdot,s)\chi)(\cdot,0,\cdot)$ is square-integrable into
$L^2((0,T)\times\Omega)$; in regime (O1) this trace is controlled, by the
mechanism of Lemma~\ref{lem:renewal-trace-compactness}, in
$L^2(0,T;H^1(\Omega))$.
\end{proposition}

\begin{proof}
Boundedness of $\mathcal T$ on $L^2(0,T)$ follows from
$\|M_\chi\|=\|\chi\|_X$, the bound \eqref{eq:Gstar-bound}, and the
boundedness of $\ell_{\bar y,\bar u}:\mathcal P\to L^2(0,T)$. The interior
representation and estimate \eqref{eq:kappa-int} are immediate from the
propagator form and Cauchy--Schwarz, and
$\int\!\!\int_{\triangle_T}|\kappa_{\rm int}|^2\le C\|\chi\|_X^2\,T\,
\|c_{\bar y}\|_{L^2(0,T;X)}^2<\infty$. The absence of a diagonal Dirac is the
statement that the evolution operator of a first-order-in-time problem is the
identity on the diagonal, not a singular measure.
\end{proof}

Proposition~\ref{prop:kernel-interior} makes the interior part of
Hypothesis~\ref{hyp:volterra-kernel} a theorem and isolates the only remaining
ingredient --- square-integrability of the propagator renewal trace --- which
is the natural backward analogue of Lemma~\ref{lem:renewal-trace-compactness}.
When the feedback does not act through the birth law ($b_{\bar y}\equiv0$, i.e.
$\partial_E \mathcal B\equiv0$), $\mathcal T=\mathcal T_{\rm int}$ and
Hypothesis~\ref{hyp:volterra-kernel} is fully established. We retain the
hypothesis form for the general case to avoid a lengthy propagator-trace
digression; the quasinilpotency conclusion below depends only on
\eqref{eq:T-kernel}, in line with the classical theory of Volterra operators
\cite{gohberg1970theory,yu2026fullcovariance}.

\begin{lemma}[Quasinilpotency, $L^2$ kernel]
\label{lem:quasinilpotent}
Under Hypothesis~\ref{hyp:volterra-kernel}, $\mathcal T$ is Hilbert--Schmidt
(hence compact \cite{brezis2011functional}) and quasinilpotent: $r(\mathcal T)=0$,
$\sigma(\mathcal T)=\{0\}$ \cite{gohberg1970theory}.
\end{lemma}

\begin{proof}
$\mathcal T$ is Hilbert--Schmidt with $\|\mathcal T\|\le\|\mathcal T\|_{\rm
HS}=\nu:=\|\kappa\|_{L^2(\triangle_T)}$; assume $\nu>0$. Fix $n\ge1$. Since
$F(\tau):=\iint_{\{0\le t\le s\le\tau\}}|\kappa|^2$ is continuous,
nondecreasing, $F(0)=0$, $F(T)=\nu^2$, choose
$\tau_j=\inf\{\tau:F(\tau)\ge(j/n)\nu^2\}$ so that
$F(\tau_j)-F(\tau_{j-1})=\nu^2/n$. Let $P_j$ be multiplication by
$\mathbf1_{(\tau_{j-1},\tau_j]}$ and $\mathcal T_{ij}:=P_i\mathcal T P_j$. By
triangularity $\mathcal T_{ij}=0$ for $i>j$, and for diagonal blocks
$\|\mathcal T_{jj}\|\le\|\mathcal T_{jj}\|_{\rm HS}\le(F(\tau_j)-
F(\tau_{j-1}))^{1/2}=\nu/\sqrt n$, while every block obeys
$\|\mathcal T_{ij}\|\le\nu$. Expanding $\mathcal T^N=\sum P_i\mathcal T^N P_j$
over nondecreasing chains $i=k_0\le\cdots\le k_N=j$, each chain has at least
$N-(n-1)$ diagonal factors, so its norm is
$\le\nu^N n^{-(N-n+1)/2}$; with at most $\binom{N+n}{n}$ chains and $n^2$ index
pairs,
\[
\|\mathcal T^N\|\le n^2\binom{N+n}{n}\nu^N n^{-(N-n+1)/2}.
\]
Taking $N$-th roots ($n$ fixed) gives $r(\mathcal T)\le\nu/\sqrt n$; letting
$n\to\infty$, $r(\mathcal T)=0$.
\end{proof}

As visualized in Figure~\ref{fig:backward_volterra}, the backward-triangular kernel support illustrates the causality collapse and the Fredholm alternative for the feedback-corrected adjoint. This geometric structure makes the quasinilpotency of the Hilbert--Schmidt operator $\mathcal{T}$ immediately intuitive, thereby establishing unconditional solvability.

\begin{figure}[htbp]
    \centering
    \includegraphics[width=\textwidth]{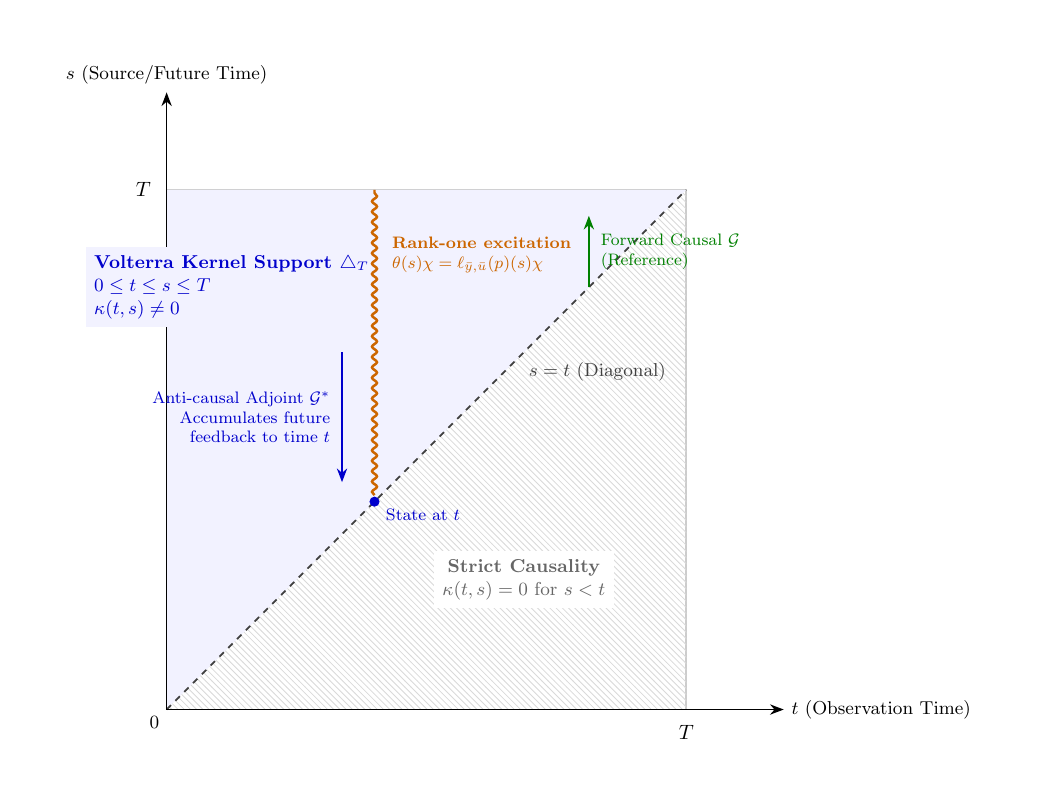} 
    \caption{The Backward Volterra adjoint and low-rank perturbation. The diagram displays the backward-triangular integration domain $\triangle_T = \{0 \le t \le s \le T\}$, contrasting the forward causal reduced operator $\mathcal{G}$ with the anti-causal adjoint $\mathcal{G}^*$. A localized rank-one excitation depicts the structural correction $\ell_{\bar{y},\bar{u}}(p)(s)\chi$, visually demonstrating why the strict triangular support eliminates diagonal singularities and guarantees the quasinilpotency of the feedback operator $\mathcal{T}$.}
    \label{fig:backward_volterra}
\end{figure}

\begin{lemma}[Quantitative bound, $L^\infty$ kernel]
\label{lem:quasinilpotent-Linf}
If in addition $\kappa\in L^\infty(\triangle_T)$ with
$\beta=\|\kappa\|_{L^\infty}$, then
$\|\mathcal T^n\|\le(\beta T)^n/n!$ and
$\|(I-\mathcal T)^{-1}\|\le e^{\beta T}$.
\end{lemma}

\begin{proof}
The iterated kernels satisfy $|\kappa_n(t,s)|\le\beta^n(s-t)^{n-1}/(n-1)!$ by
induction, so both Schur sums are $\le\beta^nT^n/n!$ and the Schur test gives
$\|\mathcal T^n\|\le(\beta T)^n/n!$; summing yields the resolvent bound.
\end{proof}

\begin{theorem}[Volterra-kernel low-rank feedback alternative]
\label{thm:low-rank-feedback-alternative}
Assume Proposition~\ref{prop:reduced-adjoint-solvability} (or its (O2)/(O3)
analogue), that $p\mapsto\ell_{\bar y,\bar u}(p)$ is bounded
$\mathcal P\to L^2(0,T)$, and Hypothesis~\ref{hyp:volterra-kernel}. Then
$I-\mathcal T$ is an isomorphism of $L^2(0,T)$, with
$(I-\mathcal T)^{-1}=\sum_{n\ge0}\mathcal T^n$, and
\eqref{eq:feedback-corrected-adjoint-abstract} has the unique solution
\begin{equation}
\label{eq:operator-valued-feedback-formula}
p=p_{\rm red}
+\mathcal G^*\!\left(\big[(I-\mathcal T)^{-1}\ell_{\bar y,\bar u}(p_{\rm red})\big](t)\chi\right).
\end{equation}
Under Lemma~\ref{lem:quasinilpotent-Linf},
$\|p-p_{\rm red}\|_{\mathcal P}\le C\,\|\ell_{\bar y,\bar u}\|\,\|\chi\|\,
e^{\beta T}\,\|\ell_{\bar y,\bar u}(p_{\rm red})\|_{L^2(0,T)}$.
\end{theorem}

\begin{proof}
By Lemma~\ref{lem:quasinilpotent}, $r(\mathcal T)=0$, so $I-\mathcal T$ is
invertible and the Neumann series converges. From
\eqref{eq:p-corrected-preliminary} and \eqref{eq:I-minus-T},
$\theta=(I-\mathcal T)^{-1}\ell_{\bar y,\bar u}(p_{\rm red})$, giving
\eqref{eq:operator-valued-feedback-formula}; conversely this $p$ solves
\eqref{eq:feedback-corrected-adjoint-abstract}. Uniqueness follows since
$(I-\mathcal T)(\theta_1-\theta_2)=0$ forces $\theta_1=\theta_2$. The
quantitative bound is Lemma~\ref{lem:quasinilpotent-Linf}.
\end{proof}

\begin{remark}[Resonance is a collapse of causality]
\label{rem:resonance-collapse}
In the time-dependent model $\mathcal T$ is strictly backward-triangular and
carries no diagonal Dirac mass, so $\sigma(\mathcal T)=\{0\}$. The
stationary/algebraic reduction collapses the time direction, $\mathcal T$
degenerates to scalar multiplication by $\ell(\psi)$
(kernel $\rightsquigarrow\ell(\psi)\delta(t-s)$), and the spectrum moves to
$\{\ell(\psi)\}$; only then can the denominator $1-\ell(\psi)$ vanish.
\end{remark}

We record the scalar Sherman--Morrison reduction.
If the feedback is scalar, $\ell:\mathcal P\to\R$, then with
$\psi:=\mathcal G^*\chi$ and $1-\ell(\psi)\neq0$,
$\displaystyle p=p_{\rm red}+\frac{\ell(p_{\rm red})}{1-\ell(\psi)}\psi$ .
This is a reduction of \eqref{eq:operator-valued-feedback-formula}, valid only
when the loop is genuinely scalar; in the full model $I-\mathcal T$ is always
invertible (Theorem~\ref{thm:low-rank-feedback-alternative}).

\subsection{Adjoint identity and the reduced gradient}
\label{subsec:adjoint-gradient-identity}
For an admissible $h$, $z=S'(\bar u)h$, and $p$ solving
\eqref{eq:feedback-corrected-adjoint-abstract}, the feedback terms cancel and
\begin{equation}
\label{eq:adjoint-gradient-identity}
\langle q_{\rm run}+q_{\rm obs},z\rangle
=
-
\int_Q\big(\mathcal K(h)\bar y\big)\cdot p\,dx\,da\,dt
=
-\sum_{\ell=1}^m\int_Q h_\ell\,(K_\ell\bar y)\cdot p .
\end{equation}
Hence the state part of the reduced gradient is represented by the density
\begin{equation}
\label{eq:state-gradient-representation}
\mathcal K^*(\bar y,p)
=
-\left(
(K_1\bar y)\cdot p,\ldots,(K_m\bar y)\cdot p
\right)^\top ,
\end{equation}
consistent with the duality
$\int_Q\mathcal K^*(\bar y,p)\cdot h=-\int_Q(\mathcal K(h)\bar y)\cdot p$. In
regime (O1), $\bar y,p\in L^2(0,T;V)\cap C([0,T];X)$, so
$(K_\ell\bar y)\cdot p\in L^1(Q)$ by the Cauchy--Schwarz inequality; thus
$\mathcal K^*(\bar y,p)\in L^1(Q;\R^m)$ in general (and in $L^2(Q)$ under the
regularity of Proposition~\ref{prop:Kstar-L2}).

\section{Optimality system, switching-function perturbation, and examples}
\label{sec:optimality-switching-examples}

\subsection{The reduced problem and a transport--diffusion compactness theorem}
\label{subsec:reduced-ocp}
For $u\in U_{\rm ad}$ write $y_u=S(u)$, $E_u=E_{y_u}$. The objective is
$\displaystyle J(y,u)
:=
J_{\rm obs}(y)
+
\frac{\gamma}{2}
\int_Q|C_Iy-y_d|^2
+
\frac{\alpha}{2}
\int_Q|u|^2 $,
with $J_{\rm obs}$ the default time-averaged surveillance (O1), and
$j(u):=J(S(u),u)$. The problem is $\min_{u\in U_{\rm ad}}j(u)$.

The compactness needed for the direct method is proved, not assumed. We first
record the renewal-trace compactness obtained by age-averaging.

\begin{lemma}[Outflow trace bound]
\label{lem:outflow-trace}
Every weak solution \(y\) obtained in Theorem~\ref{thm:closed-loop-state-wp} satisfies
$\displaystyle y(\cdot, a_{\max}, \cdot) \in L^2(0,T; L^2(\Omega;\R^n))$.
\end{lemma}

\begin{proof}
The bound follows from standard energy estimates for the transport--diffusion equation. Formally, taking the inner product of the state equation with \(y\), integrating over \((0,t) \times \mathcal I_a \times \Omega\), and applying integration by parts, we obtain
\begin{align*}
&\frac{1}{2} \|y(t)\|_{L^2(\mathcal I_a\times\Omega)}^2 + \frac{1}{2} \int_0^t \|y(s, a_{\max}, \cdot)\|_{L^2(\Omega)}^2 \,ds 
+ \int_0^t \int_{\mathcal I_a} \int_\Omega \Sigma \nabla_x y \cdot \nabla_x y \,dx\,da\,ds \\
&\quad = \frac{1}{2} \|y_0\|_{L^2(\mathcal I_a\times\Omega)}^2 + \frac{1}{2} \int_0^t \|y(s, 0, \cdot)\|_{L^2(\Omega)}^2 \,ds 
- \int_0^t \int_{\mathcal I_a} \int_\Omega C_k y^2 \,dx\,da\,ds \\
&\qquad + \int_0^t \int_{\mathcal I_a} \int_\Omega F y \,dx\,da\,ds.
\end{align*}
By the renewal boundary condition, the inflow trace \(\|y(s, 0, \cdot)\|_{L^2(\Omega)}^2\) is bounded by \(C \|y(s)\|_{L^2(\mathcal I_a\times\Omega)}^2\). Using the uniform ellipticity of \(\Sigma\) and applying Gr\"onwall's inequality yields the uniform boundedness of \(y\) in \(L^\infty(0,T; L^2(\mathcal I_a\times\Omega))\). Consequently, the left-hand side term \(\frac{1}{2} \int_0^T \|y(s, a_{\max}, \cdot)\|_{L^2(\Omega)}^2 \,ds\) is bounded, which implies \(y(\cdot, a_{\max}, \cdot) \in L^2(0,T; L^2(\Omega;\R^n))\).
\end{proof}

\begin{lemma}[$H^1$-regularity of the feedback variable]
\label{lem:Ek-H1}
Under Assumption~\ref{ass:compactness-data}, if $\{u_k\}\subset U_{\rm ad}$,
$y_k=S(u_k)$, $E_k=E_{y_k}$, then $\{E_k\}$ is bounded in $H^1(0,T)$, hence
relatively compact in $C([0,T])$.
\end{lemma}

\begin{proof}
$E_k'(t)=(\chi,\partial_t y_k)_X=\langle D_{t,a}y_k,\chi\rangle_{V',V}
-\int_\Omega\int_{\mathcal I_a}\partial_a y_k\,\chi$. The first term is bounded in
$L^2(0,T)$ by $\|D_{t,a}y_k\|_{L^2(V')}\|\chi\|_V$. Integrating by parts in
$a$, $\int_{\mathcal I_a}\partial_a y_k\,\chi=[y_k\chi]_{a=0}^{a_{\max}}
-\int_{\mathcal I_a} y_k\partial_a\chi$, where $\partial_a\chi\in X$ (by
(H-$\chi$)) bounds the interior term in $L^\infty(0,T)$. For the boundary
terms, we use $\chi(0,\cdot),\chi(a_{\max},\cdot)\in L^2(\Omega)$ together with the
traces of $y_k$. The inflow trace at $a=0$ is bounded in $L^2((0,T)\times\Omega)$ 
by the definition of the solution space $\mathcal Y$, and the outflow trace at 
$a=a_{\max}$ is bounded in $L^2((0,T)\times\Omega)$ by virtue of Lemma~\ref{lem:outflow-trace}. 
Hence $E_k'\in L^2(0,T)$ uniformly.
\end{proof}

\begin{lemma}[Renewal-trace compactness]
\label{lem:renewal-trace-compactness}
Let \(\{y_k\}\) be a sequence of states uniformly bounded in \(L^2(0,T; L^2(\mathcal I_a; H^1(\Omega)))\) and let \(\rho_k(t,x) = \int_{\mathcal I_a} \mathcal B(E_k(t), \alpha, x) y_k(t, \alpha, x) \, d\alpha\). Then, under Assumption~\ref{ass:compactness-data}, the sequence \(\{\rho_k\}\) is relatively compact in \(L^2(0,T; L^2(\Omega))\).
\end{lemma}

\begin{proof}
To apply Simon's compactness theorem (or Aubin--Lions lemma) to \(\{\rho_k\}\), we need to establish control over its time derivative \(\partial_t \rho_k\) in a suitable negative Sobolev space, specifically \(L^2(0,T; H^{-1}(\Omega))\). Differentiating \(\rho_k\) with respect to \(t\) yields
\[
\partial_t \rho_k(t,x) = \int_{\mathcal I_a} \partial_t \mathcal B(E_k(t), \alpha, x) y_k \, d\alpha + \int_{\mathcal I_a} \mathcal B(E_k(t), \alpha, x) \partial_t y_k \, d\alpha .
\]
The first term on the right-hand side is readily bounded in \(L^2(0,T; L^2(\Omega))\) due to the assumed smoothness of \(\mathcal B\) with respect to \(E\) and the boundedness of \(\partial_t E_k\). 

For the second term, we substitute the state relation \(\partial_t y_k = -\partial_a y_k + \nabla_x \cdot (\Sigma \nabla_x y_k) - C_k y_k + F_k\). The critical part is to justify the integrated diffusion term
\[
G_k(t,x) := \int_{\mathcal I_a} \mathcal B(E_k(t), \alpha, x) \nabla_x \cdot (\Sigma \nabla_x y_k(t, \alpha, x)) \, d\alpha
\]
as a bounded sequence in \(L^2(0,T; H^{-1}(\Omega))\). To this end, let \(\phi \in H^1_0(\Omega)\) be an arbitrary test function. Using the duality pairing between \(H^{-1}(\Omega)\) and \(H^1_0(\Omega)\) and applying integration by parts in \(x\), we obtain
\begin{align*}
\langle G_k(t), \phi \rangle_{H^{-1}, H^1_0} 
&= \int_{\mathcal I_a} \langle \nabla_x \cdot (\Sigma \nabla_x y_k), \mathcal B(E_k, \alpha, \cdot)\phi \rangle_{H^{-1}, H^1_0} \, d\alpha \\
&= -\int_{\mathcal I_a} \int_\Omega \Sigma \nabla_x y_k \cdot \nabla_x \big( \mathcal B(E_k, \alpha, x) \phi(x) \big) \, dx \, d\alpha \\
&= -\int_{\mathcal I_a} \int_\Omega \Sigma \nabla_x y_k \cdot \big( \phi \nabla_x \mathcal B + \mathcal B \nabla_x \phi \big) \, dx \, d\alpha .
\end{align*}
By applying H\"{o}lder's inequality and using the fact that \(\Sigma \in L^\infty(\Omega)\) and \(\mathcal B(E, \alpha, \cdot) \in W^{1,\infty}(\Omega)\) uniformly (Assumption~\ref{ass:compactness-data}), we can estimate:
\begin{align*}
\left| \langle G_k(t), \phi \rangle_{H^{-1}, H^1_0} \right| 
&\le \|\Sigma\|_{L^\infty} \|\mathcal B\|_{W^{1,\infty}} \int_{\mathcal I_a} \|\nabla_x y_k(t,\alpha,\cdot)\|_{L^2(\Omega)} \big( \|\phi\|_{L^2(\Omega)} + \|\nabla_x \phi\|_{L^2(\Omega)} \big) \, d\alpha \\
&\le C \left( \int_{\mathcal I_a} \|\nabla_x y_k(t,\alpha,\cdot)\|_{L^2(\Omega)} \, d\alpha \right) \|\phi\|_{H^1_0(\Omega)} .
\end{align*}
Using the Cauchy--Schwarz inequality with respect to \(\alpha\) yields
\[
\int_{\mathcal I_a} \|\nabla_x y_k(t,\alpha,\cdot)\|_{L^2(\Omega)} \, d\alpha \le \sqrt{a_{\max}} \|y_k(t)\|_{L^2(\mathcal I_a; H^1(\Omega))} .
\]
Therefore, the \(H^{-1}(\Omega)\) norm of \(G_k(t)\) satisfies
$\displaystyle \|G_k(t)\|_{H^{-1}(\Omega)} \le C \|y_k(t)\|_{L^2(\mathcal I_a; H^1(\Omega))} $.
Squaring and integrating from \(0\) to \(T\) gives
\[
\int_0^T \|G_k(t)\|^2_{H^{-1}(\Omega)} \, dt \le C \int_0^T \|y_k(t)\|^2_{L^2(\mathcal I_a; H^1(\Omega))} \, dt = C \|y_k\|^2_{L^2(0,T; L^2(\mathcal I_a; H^1(\Omega)))} .
\]
Since \(\{y_k\}\) is uniformly bounded in \(L^2(0,T; L^2(\mathcal I_a; H^1(\Omega)))\), the sequence \(\{G_k\}\) is uniformly bounded in \(L^2(0,T; H^{-1}(\Omega))\). 

The remaining terms resulting from \(\partial_a y_k\) (via integration by parts in \(a\)) and lower-order terms are similarly shown to be bounded in \(L^2(0,T; L^2(\Omega)) \hookrightarrow L^2(0,T; H^{-1}(\Omega))\). Consequently, \(\{\partial_t \rho_k\}\) is uniformly bounded in \(L^2(0,T; H^{-1}(\Omega))\). Since \(\{\rho_k\}\) is also bounded in \(L^2(0,T; H^1(\Omega))\) via the spatial regularity of \(\mathcal B\), Aubin--Lions--Simon theorem guarantees that \(\{\rho_k\}\) is relatively compact in \(L^2(0,T; L^2(\Omega))\).
\end{proof}

The next theorem is the conditional compactness statement; it is proved by a
characteristic Duhamel representation and the compactness of the spatial
evolution family, not by an appeal to standard Aubin--Lions (which would
require control of $\partial_t y_k$, unavailable here). The structure of the
argument follows the transport--diffusion theory for age-structured systems
\cite{webb1985theory,langlais1985nonlinear,yu2026chemotactic}.

\begin{theorem}[Conditional transport--diffusion compactness]
\label{thm:transport-diffusion-compactness}
Let Assumptions~\ref{ass:diffusion}--\ref{ass:compactness-data} hold. Assume in addition that the family of states \(\{S(u):u\in U_{\rm ad}\}\) satisfies a uniform time-translation compactness estimate in \(L^2(\mathcal I_a\times \Omega)\):
\[
\lim_{\tau\downarrow0}
\sup_{u\in U_{\rm ad}}
\|S(u)(\cdot+\tau,\cdot,\cdot)-S(u)(\cdot,\cdot,\cdot)\|_{L^2(0,T-\tau;L^2(\mathcal I_a\times \Omega))}
=0 .
\]
If \(u_k\overset{*}{\rightharpoonup}u\) in \(L^\infty(Q;\R^m)\), then, after extraction of a subsequence, $S(u_k)\to S(u)$ strongly in $L^2(Q;\R^n)$.
\end{theorem}

\begin{proof}
Let \(y_k = S(u_k)\) be the sequence of states corresponding to the controls \(u_k\). Since \(u_k\) is uniformly bounded in \(L^\infty(Q;\R^m)\) (as \(u_k \overset{*}{\rightharpoonup} u\)), the standard well-posedness and a priori estimates (e.g., energy estimates along characteristics) imply that the sequence \(\{y_k\}_{k\in\N}\) is uniformly bounded in \(L^2(Q;\R^n)\).

To establish strong convergence, we utilize Simon's compactness theorem (see e.g., \cite{simon1986compact}). The spatial and age regularization is provided by the diffusion and the characteristic Volterra integration. Specifically, for any fixed \(t\) and \(a\), the evolution operator \(U(a,\sigma)\) maps into domains with higher spatial regularity (under Assumption~\ref{ass:diffusion}), granting compactness in the spatial domain \(\Omega\). 

However, as integration along characteristics \(s=t-a\) does not inherently regularize the \(s\)-variable, the spatial compactness alone is insufficient for strong compactness in \(L^2(Q;\R^n)\). This obstruction is resolved by the uniform time-translation assumption.

By the assumed uniform time-translation condition, we have:
\[
\lim_{\tau\downarrow0} \sup_{k\in\N} \|y_k(\cdot+\tau) - y_k(\cdot)\|_{L^2(0,T-\tau;L^2(\mathcal I_a\times \Omega))} = 0.
\]
Combining this time-translation equicontinuity with the uniform boundedness of \(\{y_k\}\) in \(L^2(Q;\R^n)\) and the compact embedding properties in the spatial variables provided by the diffusion operator, the conditions of Simon's theorem are satisfied. 

Therefore, \(\{y_k\}_{k\in\N}\) is relatively compact in \(L^2(0,T; L^2(\mathcal I_a\times \Omega)) \simeq L^2(Q;\R^n)\). It follows that there exists a subsequence, still denoted by \(y_k\), such that \(y_k \to y^*\) strongly in \(L^2(Q;\R^n)\). 

Finally, passing to the limit in the weak formulation of the state equation with \(u_k \overset{*}{\rightharpoonup} u\) and \(y_k \to y^*\), the bilinear terms converge in the sense of distributions. By uniqueness of the solution to the state equation, we conclude \(y^* = S(u)\), which completes the proof.
\end{proof}

\begin{theorem}[Existence of optimal controls]
\label{thm:existence-localized-observations}
Under Assumptions~\ref{ass:diffusion}--\ref{ass:compactness-data} and the
default regime (O1), the problem $\min_{U_{\rm ad}}j$ has a global minimizer.
\end{theorem}

\begin{proof}
Let $\{u_k\}$ be minimizing. By Banach--Alaoglu,
$u_k\overset{\ast}{\rightharpoonup}\bar u\in U_{\rm ad}$ (subsequence). By
Theorem~\ref{thm:transport-diffusion-compactness}, $S(u_k)\to S(\bar u)$
strongly in $L^2(Q)$, so the distributed tracking term passes to the limit and,
since each $\mathcal O_m$ is continuous on $L^2(Q)$ in regime (O1),
$J_{\rm obs}(S(u_k))\to J_{\rm obs}(S(\bar u))$. As weak$^\ast$ convergence in
$L^\infty$ implies weak convergence in $L^2$ on bounded sets,
$\|\bar u\|_{L^2}^2\le\liminf_k\|u_k\|_{L^2}^2$, so
$j(\bar u)\le\liminf_k j(u_k)=\inf j$.
\end{proof}

\begin{corollary}[Existence: conditional regimes]
\label{cor:existence-point}
In regime (O2), if additionally $\{S(u_k)\}$ is relatively compact in
$C([0,T];X)$, or in regime (O3) if $\{S(u_k)\}$ is relatively compact in a
space compactly embedded in $C(\overline Q;\R^n)$, then a global minimizer
exists, and in (O3) $\mathcal K^*(\bar y,\bar p)\in L^2(Q)$.
\end{corollary}

\begin{proof}
The extra compactness makes $J_{\rm obs}$ continuous in the relevant topology;
the proof of Theorem~\ref{thm:existence-localized-observations} applies. In
(O3), $\bar p\in\mathcal P_\sharp\hookrightarrow L^2(Q)$ gives the stated
integrability.
\end{proof}

\subsection{Reduced gradient and first-order optimality system}
\label{subsec:gradient-first-order}
Let $\bar u$ be a local minimizer (in $L^\infty(Q;\R^m)$), $\bar y=S(\bar u)$,
and $\bar p$ the feedback-corrected adjoint of
Theorem~\ref{thm:low-rank-feedback-alternative}. Define
$\mathcal K^*(\bar y,\bar p)$ by \eqref{eq:state-gradient-representation}; in
regime (O1) it lies in $L^1(Q;\R^m)$.

\begin{lemma}[Derivative of the reduced cost]
\label{lem:reduced-gradient-section4}
For $h\in L^\infty(Q;\R^m)$,
$\displaystyle j'(u)h
=
\int_Q
\left(
\alpha u+\mathcal K^*(y,p)
\right)\cdot h $,
the integral being well defined because $\alpha u\in L^\infty$,
$\mathcal K^*(y,p)\in L^1$, and $h\in L^\infty$.
\end{lemma}

\begin{proof}
With $z=S'(u)h$, $j'(u)h=\langle q_{\rm run}+q_{\rm obs},z\rangle+\alpha\int_Q
u\cdot h$, and \eqref{eq:adjoint-gradient-identity} with
\eqref{eq:state-gradient-representation} gives $\langle q_{\rm run}+q_{\rm
obs},z\rangle=\int_Q\mathcal K^*(y,p)\cdot h$.
\end{proof}

The reduced derivative is first obtained as a functional on
\(L^\infty(Q;\R^m)\). Thus the natural gradient density belongs only to
\(L^1(Q;\R^m)\) in the default energy regime. The projection formula is
therefore understood pointwise a.e.; it is not an \(L^2\)-gradient identity
unless the additional boundedness of Proposition~\ref{prop:Kstar-L2} holds.

\begin{theorem}[First-order necessary condition]
\label{thm:first-order-section4}
Let $\bar u$ be a local minimizer. Then
\begin{equation}
\label{eq:variational-inequality-section4}
\int_Q
\left(
\alpha\bar u+\mathcal K^*(\bar y,\bar p)
\right)\cdot(u-\bar u)
\ge0
\qquad
\forall u\in U_{\rm ad},
\end{equation}
the integrand lying in $L^1(Q)$ since $u-\bar u\in L^\infty$. Equivalently,
componentwise and a.e.,
\begin{equation}
\label{eq:projection-formula-section4}
\bar u_\ell
=
\Pi_{[u_{\min,\ell},u_{\max,\ell}]}
\left(
-\alpha^{-1}\mathcal K_\ell^*(\bar y,\bar p)
\right),
\qquad
\ell=1,\ldots,m .
\end{equation}
\end{theorem}

\begin{proof}
Convexity of $U_{\rm ad}$ and Lemma~\ref{lem:reduced-gradient-section4} give
\eqref{eq:variational-inequality-section4}. Localizing the inequality with
$h=\mathbf1_S e_\ell\in L^\infty$ over measurable $S$ yields the pointwise
complementarity, hence \eqref{eq:projection-formula-section4} \cite{troltzsch2010optimal}.
The right-hand side of \eqref{eq:projection-formula-section4} is automatically
in $L^\infty(Q)$, since $\Pi_{[u_{\min,\ell},u_{\max,\ell}]}$ maps any
real-valued (here $L^1$) density into $[u_{\min,\ell},u_{\max,\ell}]$
pointwise; thus the projection formula is meaningful even when
$\mathcal K^*(\bar y,\bar p)$ is only $L^1(Q)$.
\end{proof}

\begin{proposition}[$L^2$ gradient under boundedness]
\label{prop:Kstar-L2}
If in addition $\bar y\in L^\infty(Q;\R^n)$ (e.g.\ in the smoothing regime of
Remark~\ref{rem:point-sufficient}, or when $y^0,F$ are bounded and a maximum
principle applies), then $\mathcal K^*(\bar y,\bar p)\in L^2(Q;\R^m)$, and the
variational inequality and switching analysis hold in $L^2(Q)$.
\end{proposition}

\begin{proof}
$(K_\ell\bar y)\cdot\bar p$ is bounded by $\|K_\ell\|_\infty\|\bar
y\|_{L^\infty(Q)}|\bar p|$ pointwise, and $\bar p\in L^2(Q)$, so the product is
in $L^2(Q)$.
\end{proof}

We set $S_{\rm sw}:=\alpha\bar u+\mathcal K^*(\bar y,\bar p)$, with the usual
complementarity \eqref{eq:variational-inequality-section4}.

\subsection{Feedback perturbation of the switching function}
\label{subsec:switching-perturbation}
With $p_{\rm red}=\mathcal G^*(q_{\rm run}+q_{\rm obs})$ and
\[
p_{\rm fb}
:=
\mathcal G^*\!\left(\big[(I-\mathcal T)^{-1}\ell_{\bar y,\bar u}(p_{\rm red})\big](t)\chi\right),
\qquad
\bar p=p_{\rm red}+p_{\rm fb},
\]
linearity of $\mathcal K^*(\bar y,\cdot)$ gives
$S_{\rm sw}=S_{\rm red}+S_{\rm fb}$ with
$S_{\rm red}=\alpha\bar u+\mathcal K^*(\bar y,p_{\rm red})$ and
$S_{\rm fb}=\mathcal K^*(\bar y,p_{\rm fb})$. (All densities lie in $L^1(Q)$,
or $L^2(Q)$ under Proposition~\ref{prop:Kstar-L2}.) In the scalar reduction,
$S_{\rm fb}=\frac{\ell(p_{\rm red})}{1-\ell(\psi)}\mathcal K^*(\bar y,\psi)$.

\begin{theorem}[Threshold preservation under weak feedback]
\label{thm:threshold-preservation}
Let $m=1$ and $S_{\rm sw}=S_{\rm red}+S_{\rm fb}$. If, for a.e.\ $(t,a,x)$ in
$Q\setminus\mathcal Z_\delta$ with
$\mathcal Z_\delta:=\{|S_{\rm red}|\le\delta\}$, one has
$|S_{\rm fb}|<|S_{\rm red}|$, then $\sgn S_{\rm sw}=\sgn S_{\rm red}$ there.
Moreover the decision-change set
$\mathcal C_{\rm fb}:=\{S_{\rm red}S_{\rm sw}<0\}$ satisfies
$\mathcal C_{\rm fb}\subset\{|S_{\rm red}|\le|S_{\rm fb}|\}$.
\end{theorem}

\begin{proof}
If $S_{\rm red}>0$ then $S_{\rm sw}\ge S_{\rm red}-|S_{\rm fb}|>0$; if
$S_{\rm red}<0$ then $S_{\rm sw}\le S_{\rm red}+|S_{\rm fb}|<0$. If
$S_{\rm red}S_{\rm sw}<0$ then $|S_{\rm fb}|\ge|S_{\rm red}|$.
\end{proof}

As depicted in Figure~\ref{fig:threshold_preservation}, the feedback perturbation alters the optimal intervention decision by translating the abstract functional variational inequality into a clear phase-diagram for the control constraint sets. Specifically, it graphically demonstrates Theorem~\ref{thm:threshold-preservation}, showing that the optimal bang-bang control sign is robust outside the narrow margin $|S_{\rm red}| \le |S_{\rm fb}|$.

\begin{figure}[htbp]
    \centering
    \includegraphics[width=\textwidth]{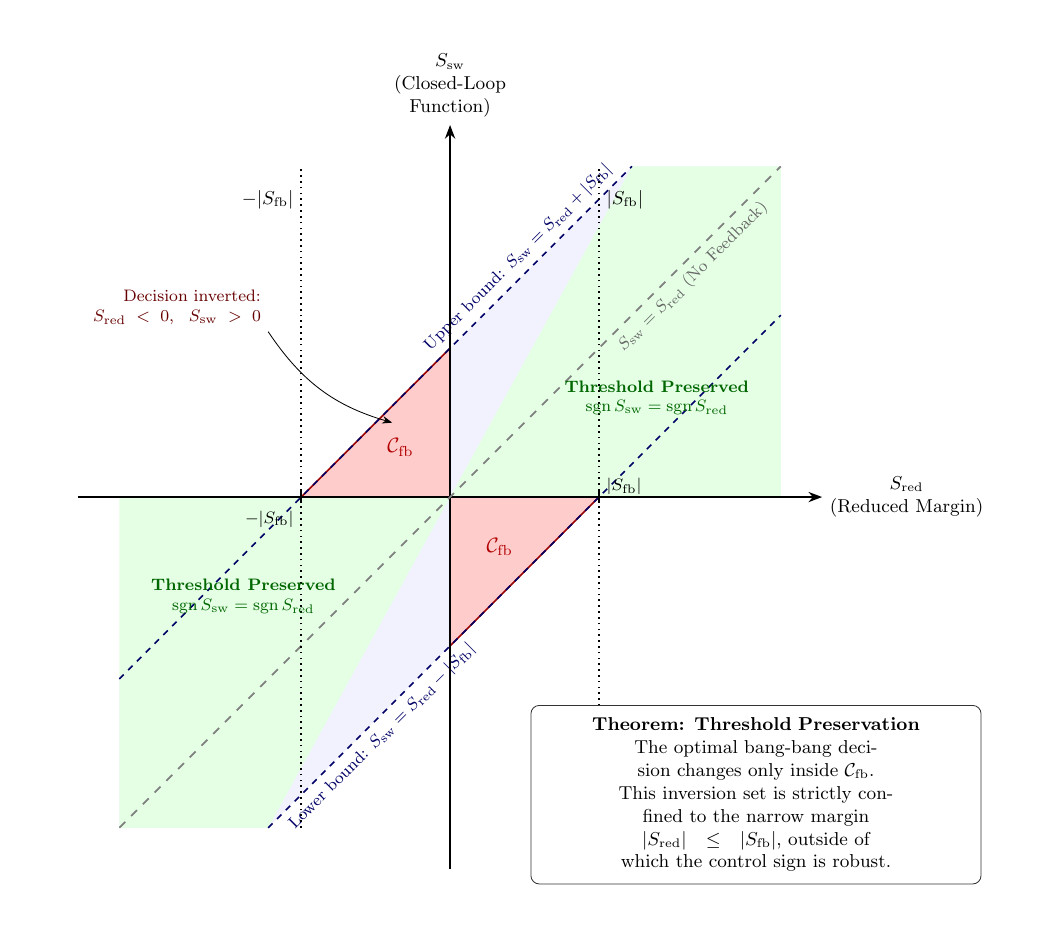} 
    \caption{Threshold preservation in the switching function. 
    This graphical representation of Theorem~\ref{thm:threshold-preservation} illustrates how the feedback perturbation alters the optimal intervention decision. The uncertainty band is defined by $\pm |S_{\rm fb}|$. The shaded decision-change set $\mathcal{C}_{\rm fb}$ is strictly confined inside the region where $|S_{\rm red}| \le |S_{\rm fb}|$. Outside this narrow margin, the sign of the switching function (and thus the optimal bang-bang control) is perfectly preserved.}
    \label{fig:threshold_preservation}
\end{figure}

\subsection{Examples of switching functions}
\label{subsec:switching-examples}
\paragraph{Vaccination transfer.}
With $y=(S,V,I,R)^\top$, $p=(p_S,p_V,p_I,p_R)^\top$, and (left-hand-side
convention) $K_{\rm vac}y=(S,-S,0,0)^\top$, one has
$\mathcal K_{\rm vac}^*(y,p)=-(K_{\rm vac}y)\cdot p=S(p_V-p_S)$, so
$S_{\rm sw}^{\rm vac}=\alpha u+S(p_V-p_S)$,
$\bar u_{\rm vac}=\Pi_{[u_{\min},u_{\max}]}(-\alpha^{-1}\bar S(\bar p_V-\bar
p_S))$, and $S_{\rm fb}^{\rm vac}=\bar S((p_{\rm fb})_V-(p_{\rm fb})_S)$.

\paragraph{Removal of infected.}
With $K_{\rm rem}y=(0,0,I,0)^\top$,
$\mathcal K_{\rm rem}^*(y,p)=-I p_I$, so
$S_{\rm sw}^{\rm rem}=\alpha u-I p_I$,
$\bar u_{\rm rem}=\Pi_{[u_{\min},u_{\max}]}(\alpha^{-1}\bar I\bar p_I)$, and
$S_{\rm fb}^{\rm rem}=-\bar I(p_{\rm fb})_I$.

\begin{remark}[Sign convention]
If the control term is placed on the right-hand side, the signs reverse. The
invariant rule is
$\int_Q\mathcal K^*(y,p)\cdot h=-\int_Q(\mathcal K(h)y)\cdot p$.
\end{remark}

\subsection{A minimal feedback-diagnostic experiment}
\label{subsec:minimal-feedback-diagnostic}
For a computed locally optimal pair $(\bar y,\bar u)$:
\begin{enumerate}[label=\textnormal{(\roman*)},leftmargin=*]
\item solve $\mathscr A_{\rm red}^*p_{\rm red}=q_{\rm run}+q_{\rm obs}$;
\item compute $\theta=(I-\mathcal T)^{-1}\ell_{\bar y,\bar u}(p_{\rm red})$ and
$p_{\rm fb}=\mathcal G^*(\theta(t)\chi)$;
\item form $S_{\rm red}=\alpha\bar u+\mathcal K^*(\bar y,p_{\rm red})$,
$S_{\rm fb}=\mathcal K^*(\bar y,p_{\rm fb})$, $S_{\rm sw}=S_{\rm red}+S_{\rm
fb}$;
\item identify $\mathcal C_{\rm fb}=\{S_{\rm red}S_{\rm sw}<0\}\subset
\{|S_{\rm red}|\le|S_{\rm fb}|\}$.
\end{enumerate}
In the scalar stationary case, $p_{\rm fb}=\frac{\ell(p_{\rm
red})}{1-\ell(\psi)}\psi$ and the loop-gain margin
$\Delta_{\rm fb}=|1-\ell(\psi)|$ controls the size of $S_{\rm fb}$. By
contrast, in the full time-dependent model
(Theorem~\ref{thm:low-rank-feedback-alternative}) no such margin can vanish:
$I-\mathcal T$ is always invertible.

\appendix
\section{Technical Details Supporting the Main Analysis}
\label{app:technical-details}

This appendix gives the derivation details behind three points used in the
main text: compactness without a separate \(\partial_t y\)-estimate, the
feedback-corrected adjoint, and the switching-function perturbation formula.
We write the arguments in a compressed form and use display formulas only when
they clarify the structure.

\subsection{Why the standard Lions--Magenes argument is not used}
\label{app:lions-magenes}

For a parabolic evolution equation in a Gelfand triple \(V\hookrightarrow
X\hookrightarrow V'\), the usual energy setting is \(y\in L^2(0,T;V)\) and
\(\partial_t y\in L^2(0,T;V')\). Under these two estimates one obtains
\(y\in C([0,T];X)\) by the Lions--Magenes continuity theorem. The present
age--space equation does not give this pair of estimates. The natural
derivative is the transport derivative \(D_{t,a}y:=\partial_t y+\partial_a y\),
and the energy bound gives \(y\in L^2(0,T;V)\) and \(D_{t,a}y\in
L^2(0,T;V')\). Since \(\partial_t y=D_{t,a}y-\partial_a y\), a bound for
\(D_{t,a}y\) does not imply a bound for \(\partial_t y\) unless
\(\partial_a y\) is separately controlled. Such a separate age-derivative
estimate is not available in the transport energy space. Therefore the
standard parabolic continuity route cannot be invoked directly.

The replacement is to straighten the transport direction. Fix
\(s=t-a\) and define \(Y_s(a,x):=y(s+a,a,x)\). Then, by the chain rule,
\(\partial_aY_s(a,x)=(\partial_t y+\partial_a y)(s+a,a,x)=D_{t,a}y(s+a,a,x)\).
Thus, along the characteristic \(t-a=s\), the age--time transport operator
becomes an ordinary evolution derivative in \(a\), while the spatial diffusion
remains parabolic. The equation becomes, schematically,
\[
\partial_aY_s-\nabla_x\cdot(\Sigma(a,x)\nabla_xY_s)+C_s(a,x)Y_s=G_s(a,x).
\]
If \(U(a,\sigma)\) denotes the spatial evolution family, then for
\(a\ge a_0\) the characteristic mild formula is
\[
Y_s(a)=U(a,a_0)Y_s(a_0)+\int_{a_0}^aU(a,\sigma)
\{G_s(\sigma)-C_s(\sigma)Y_s(\sigma)\}\,d\sigma .
\]
This formula gives continuity in the evolution variable \(a\) and hence
continuity of \(t\mapsto y(t)\) in \(X\), after the initial and renewal inflow
pieces are inserted. The conclusion \(y\in C([0,T];X)\) is therefore obtained
from the characteristic mild representation, not from a Lions--Magenes
argument based on \(\partial_t y\).

\subsection{Weak-star controls and strong convergence of states}
\label{app:weakstar-strong}

Let \(u_k\overset{*}{\rightharpoonup}u\) in \(L^\infty(Q;\R^m)\), and let
\(y_k=S(u_k)\). The goal is to extract a subsequence such that \(y_k\to y\)
strongly in \(L^2(Q;\R^n)\) and then identify \(y=S(u)\).

The characteristic mild representation has the schematic form
\(y_k=\Psi_k-\mathcal V(C_ky_k)\), where \(C_k=\mathcal A(E_k)+\mathcal N[y_k]
+\mathcal K(u_k)\). Here \(\Psi_k\) contains the initial inflow, the renewal
inflow, and the source term. The a priori estimate gives \(y_k\rightharpoonup
y\) in \(L^2(0,T;V)\), while the feedback compactness gives \(E_k\to E\) in
\(C([0,T])\). The renewal trace compactness gives \(\Psi_k\to\Psi\) strongly
in \(L^2(Q)\).

Next, use the compactness of \(\mathcal V\). If \(g_k\rightharpoonup g\) in
\(L^2(Q)\), then \(\mathcal Vg_k\to\mathcal Vg\) strongly in \(L^2(Q)\). Apply
this to \(g_k=C_ky_k\), using boundedness of \(C_k\) in \(L^\infty\). The
coefficient pieces pass as follows. Since \(E_k\to E\) uniformly,
\(\mathcal A(E_k)y_k\rightharpoonup \mathcal A(E)y\). Since the nonlocal map is compact,
\(\mathcal N[y_k]\to\mathcal N[y]\) strongly and hence
\(\mathcal N[y_k]y_k\rightharpoonup\mathcal N[y]y\). Finally, for each control
component, \(u_{k,\ell}\overset{*}{\rightharpoonup}u_\ell\) in \(L^\infty\);
once strong convergence of \(y_k\) is obtained, \(K_\ell y_k\to K_\ell y\)
strongly in \(L^2\), and therefore \(u_{k,\ell}K_\ell y_k\rightharpoonup
u_\ell K_\ell y\) weakly in \(L^2\).

To obtain the strong convergence, use the Volterra resolvent form
\(y_k=(I+\mathcal V_k^C)^{-1}\Psi_k\), where
\(\mathcal V_k^Cg:=\mathcal V(C_kg)\). The age-Volterra structure yields the
uniform estimate \(\|(\mathcal V_k^C)^j\|\le (CMa_{\max})^j/j!\). Hence the
Neumann series for \((I+\mathcal V_k^C)^{-1}\) converges uniformly in \(k\).
For each fixed \(j\), compactness of \(\mathcal V\) and the weak convergence
of \(C_kg\) imply \((\mathcal V_k^C)^j\Psi_k\to(\mathcal V^C)^j\Psi\) strongly.
The uniform convergence of the Neumann series then gives \(y_k\to y\) strongly
in \(L^2(Q)\). Passing to the limit in the weak formulation identifies \(y\)
as the solution driven by \(u\), so \(y=S(u)\) by uniqueness.

\subsection{Derivation of the feedback-corrected adjoint}
\label{app:feedback-adjoint}

Let \(z=S'(\bar u)h\) be the linearized state. The feedback observable is
\(E_y(t)=\int_{\mathcal I_a}\int_\Omega\chi\cdot y\,dx\,da\), so its variation is
\(\delta E_z(t)=\int_{\mathcal I_a}\int_\Omega\chi\cdot z\,dx\,da\). The derivative of
the interior coefficient gives the term \(\delta E_z(t)c_{\bar y}\), where
\(c_{\bar y}:=\partial_E\mathcal A(\bar E)\bar y\). The derivative of the renewal law
gives the boundary term \(\delta E_z(t)b_{\bar y}\), where
\(b_{\bar y}(t,x):=\int_{\mathcal I_a}\partial_E\mathcal B(\bar E(t),\alpha,x)
\bar y(t,\alpha,x)\,d\alpha\).

To derive the adjoint correction, pair the linearized equation with a test
adjoint \(p\). The interior feedback term contributes
\(\int_0^T\delta E_z(t)\int_{\mathcal I_a}\int_\Omega c_{\bar y}\cdot p\,dx\,da\,dt\).
The renewal feedback term enters with the opposite sign after integration by
parts in the age-transport operator and contributes
\(-\int_0^T\delta E_z(t)\int_\Omega b_{\bar y}\cdot p(t,0,x)\,dx\,dt\). Hence
the total feedback contribution equals
$\displaystyle -\int_0^T\ell_{\bar y,\bar u}(p)(t)\delta E_z(t)\,dt$,
where
\[
\ell_{\bar y,\bar u}(p)(t):=
-\int_{\mathcal I_a}\int_\Omega c_{\bar y}\cdot p\,dx\,da
+\int_\Omega b_{\bar y}\cdot p(t,0,x)\,dx .
\]
Since \(\delta E_z(t)=\int_{\mathcal I_a}\int_\Omega\chi\cdot z\,dx\,da\), the last
quantity becomes
\[
-\int_Q \ell_{\bar y,\bar u}(p)(t)\chi(a,x)\cdot z(t,a,x)\,dx\,da\,dt .
\]
Therefore the feedback derivative is represented in the adjoint equation by
the separable source \(\ell_{\bar y,\bar u}(p)(t)\chi(a,x)\). It is rank-one
in the structured variables \((a,x)\), but not scalar in time.

\subsection{Volterra feedback operator and quasinilpotency}
\label{app:volterra-quasinilpotency}

Let \(\mathcal G^*\) be the reduced backward adjoint solution operator. If
\(p_{\rm red}:=\mathcal G^*q\) is the reduced adjoint, then the full adjoint
satisfies \(p=p_{\rm red}+\mathcal G^*(\ell_{\bar y,\bar u}(p)(t)\chi)\).
Set \(\theta:=\ell_{\bar y,\bar u}(p)\). Applying \(\ell_{\bar y,\bar u}\) to
the previous identity gives \((I-\mathcal T)\theta
=\ell_{\bar y,\bar u}(p_{\rm red})\), where
$\displaystyle \mathcal T\theta:=\ell_{\bar y,\bar u}\big(\mathcal G^*(\theta(t)\chi)\big)$.

The reduced forward solution operator is causal, so its adjoint
\(\mathcal G^*\) is anti-causal. Thus \(\mathcal T\) is backward Volterra in
time: if \(P_\tau\) denotes multiplication by \(\mathbf 1_{(\tau,T]}\), then
\(P_\tau\mathcal T=P_\tau\mathcal TP_\tau\). Under the additional
propagator-kernel hypothesis, this anti-causal operator has the triangular
representation
\[
(\mathcal T\theta)(t)=\int_t^T\kappa(t,s)\theta(s)\,ds,
\qquad
\kappa\in L^2(\{0\le t\le s\le T\}).
\]
Then \(\mathcal T\) is Hilbert--Schmidt and compact. Moreover, the triangular
support gives the Volterra iteration estimate. Indeed, the \(N\)-fold kernel
is supported on \(t\le s_1\le\cdots\le s_N\le T\), and subdivision of the
triangle into \(n\) strips gives the bound
\(\|\mathcal T^N\|\le C_n\|\kappa\|_{L^2}^N n^{-N/2}\) up to a polynomial
factor in \(N\). Taking \(N\)-th roots gives
\(r(\mathcal T)\le \|\kappa\|_{L^2}/\sqrt n\), and then \(n\to\infty\) gives
\(r(\mathcal T)=0\). Hence \(\sigma(\mathcal T)=\{0\}\), and \(I-\mathcal T\)
is invertible.

Consequently, in the Volterra-kernel regime,
$\displaystyle p=p_{\rm red}
+\mathcal G^*\left(
\big[(I-\mathcal T)^{-1}\ell_{\bar y,\bar u}(p_{\rm red})\big](t)\chi
\right)$.
This is the feedback-corrected adjoint formula. It is not the scalar
Sherman--Morrison formula. The scalar denominator \(1-\ell(\psi)\) appears
only if the time-dependent Volterra loop is collapsed to a stationary or
algebraic feedback loop.

\subsection{Switching-function decomposition}
\label{app:switching}

The reduced gradient identity is
\[
j'(\bar u)h
=
\int_Q\{\alpha\bar u+\mathcal K^*(\bar y,\bar p)\}\cdot h\,dx\,da\,dt .
\]
The feedback-corrected adjoint decomposes as \(\bar p=p_{\rm red}+p_{\rm fb}\),
where
$\displaystyle p_{\rm fb}:=
\mathcal G^*\left(
\big[(I-\mathcal T)^{-1}\ell_{\bar y,\bar u}(p_{\rm red})\big](t)\chi
\right)$.
Because \(\mathcal K^*(\bar y,\cdot)\) is linear in the adjoint variable,
\(\mathcal K^*(\bar y,\bar p)=\mathcal K^*(\bar y,p_{\rm red})
+\mathcal K^*(\bar y,p_{\rm fb})\). Therefore
\(S_{\rm sw}:=\alpha\bar u+\mathcal K^*(\bar y,\bar p)\) decomposes as
\(S_{\rm sw}=S_{\rm red}+S_{\rm fb}\), where
\(S_{\rm red}:=\alpha\bar u+\mathcal K^*(\bar y,p_{\rm red})\) and
\(S_{\rm fb}:=\mathcal K^*(\bar y,p_{\rm fb})\).

For a scalar control, the decision changes only when the feedback correction
dominates the reduced margin. Indeed, if \(S_{\rm red}>0\) and
\(|S_{\rm fb}|<|S_{\rm red}|\), then \(S_{\rm sw}=S_{\rm red}+S_{\rm fb}>0\).
If \(S_{\rm red}<0\) and \(|S_{\rm fb}|<|S_{\rm red}|\), then
\(S_{\rm sw}<0\). Hence sign reversal can occur only where
$\displaystyle \mathcal C_{\rm fb}:=\{S_{\rm red}S_{\rm sw}<0\}
\subset
\{|S_{\rm red}|\le |S_{\rm fb}|\}$.
Thus the feedback loop affects the optimal intervention only near the reduced
switching surface, unless the feedback correction is itself large.

\bibliography{reference}

\end{document}